\documentclass[11pt]{article}
\usepackage{amsmath, amsfonts, amsthm, amssymb, color}
\usepackage{graphicx}
\usepackage{float}
\usepackage{verbatim}

\allowdisplaybreaks

\numberwithin{equation}{section}

\newtheorem{lemma}{Lemma}[section]

\newtheorem{proposition}[lemma]{Proposition}
\newtheorem{theorem}[lemma]{Theorem}
\newtheorem{corollary}[lemma]{Corollary}
\theoremstyle{definition}
\newtheorem{definition}[lemma]{Definition}
\newtheorem{example}[lemma]{Example}

\theoremstyle{remark}
\newtheorem{remark}[lemma]{Remark}

\newcommand{\R}{\ensuremath{\mathbb{R}}}

\newcommand{\N}{\ensuremath{\mathbb{N}}}
\renewcommand{\a}{\mathbf{a}}
\renewcommand{\b}{\mathbf{b}}
\renewcommand{\c}{\mathbf{c}}
\renewcommand{\d}{\mathbf{d}}

\def\m{\mathfrak{m}}

\DeclareMathOperator{\CP}{CP}
\DeclareMathOperator{\Det}{Det}

\newcommand{\A}{\mathcal{A}}
\newcommand{\wA}{\widehat{A}}

\renewcommand{\t}{\mathbf{t}}
\newcommand{\eps}{\varepsilon}

\renewcommand{\t}{\mathbf{t}}

\renewcommand{\d}{\mathbf{d}}

\newcommand{\p}{\mathbf{p}}

\renewcommand{\L}{\mathcal{L}}

\renewcommand{\O}{\mathcal{O}}

\numberwithin{equation}{section} \numberwithin{table}{section}

\date{\today}
\begin{document}

\title{Analogues of Khintchine's theorem for random attractors}
\author{Simon Baker$^{1}$ and Sascha Troscheit$^{2}$\\ \\
	\emph{$^{1}$School of Mathematics,} \\ \emph{University of Birmingham,} \\ \emph{Birmingham,
	B15 2TT, UK.} \\ Email: simonbaker412@gmail.com\\ \\
	\emph{$^2$Faculty of Mathematics,} \\ \emph{University of Vienna,} \\
	\emph{Oskar-Morgenstern-Platz 1, 1090 Wien, Austria}\\ Email: saschatro@gmail.com\\ \\
}

\maketitle

\begin{abstract}
In this paper we study random iterated function systems. Our main result gives sufficient conditions
for an analogue of a well known theorem due to Khintchine from Diophantine approximation to hold
almost surely for stochastically self-similar and self-affine random iterated function systems. \\

\noindent \emph{Mathematics Subject Classification 2010}: 28A80, 37C45, 60J80.\\

\noindent \emph{Key words and phrases}: Random iterated function systems, Diophantine approximation, self-similar systems,
self-affine systems, Khintchine's theorem.

\end{abstract}

\section{Introduction}\label{sect:intro}
Khintchine's theorem is an important result in number theory which demonstrates that the Lebesgue
measure of certain limsup sets defined using the rationals is determined by the
convergence/divergence of naturally occurring volume sums. Inspired by this result, the first author
studied fractal analogues of Khintchine's theorem where the role of the rationals is played by a
natural set of points that is generated by the underlying iterated function system
\cite{Bak2,Bak,Bakerapprox2,BakOver}. The results of \cite{BakOver} demonstrate that for many
parameterised families of overlapping iterated function systems, we typically observe Khintchine
like behaviour. The results of \cite{BakOver} also demonstrate that by viewing overlapping iterated
function systems through the lens of Diophantine approximation, we obtain a new meaningful framework
for classifying iterated function systems.

In this article we consider analogues of Khintchine's theorem for random models of attractors. In
particular, we investigate Khintchine type results for random recursive fractal sets, a natural
class of randomly generated sets commonly used as a model for self-similar and self-affine sets. Our
main result shows that under appropriate hypothesis for our random recursive model, we will almost
surely observe Khintchine like behaviour. To state our main result in full it is necessary to
properly formalise our model and introduce several other notions. 
To motivate what follows we include the following easier to state theorem that is a consequence of
Theorem~\ref{Main theorem}.

\begin{theorem}
	\label{baby theorem}
Let $t_1$ and $t_2$ be two distinct real numbers. Fix $r_1$ and $r_2$ satisfying $0\leq r_1<r_2<1$ and let $\eta$ be the normalised
Lebesgue measure on $[r_1,r_2]$. Let $\textbf{r}:=(r_{\a})_{\a\in \cup_{n=1}^{\infty}\{1,2\}^{n}}$
be a sequence of real numbers enumerated by the finite words with digits in $\{1,2\},$ such that each
$r_{\a}$ is chosen independently from $[r_1,r_2]$ according to the law $\eta$. For each $\textbf{r}$
we define a projection map $\Pi_{\textbf{r}}:\{1,2\}^{\N}\to\mathbb{R}^{d}$ given by
$$\Pi_{\textbf{r}}(\b):=\sum_{k=1}^{\infty}t_{b_k}\cdot \prod_{j=1}^{k-1}r_{b_{1}\ldots b_{j}}.$$
Let $\b\in \{1,2\}^{\N}$ and assume $\log 2> -\int_{r_{1}}^{r_{2}}\log r\, dr.$ Then for almost
every $\textbf{r}$ the set $$\left\{x\in \mathbb{R}^{d}:|x-\Pi_{\textbf{r}}(a_{1}\ldots
a_{m}\b)|\leq \frac{1}{2^{m}\cdot m}\textrm{ for i.m. }a_{1}\ldots a_{m}\in
\cup_{n=1}^{\infty}\A^{n} \right\}$$ has positive Lebesgue measure.
\end{theorem}
Here and throughout we write i.m.~as a shorthand for infinitely many. We emphasise that our main
result also covers higher dimensional random iterated function systems that may contain affine maps.

\section{Background}
\label{sect:background}
In this section we recall some background results from fractal geometry and Diophantine
approximation. We also detail our motivating problem in the deterministic setting and provide some
background on random models for iterated function systems.

\subsection{Fractal Geometry}
Given a finite set of contractions $\Phi=\{\phi_i:\R^d\to\R^d\}_{i\in \A}$
there exists, by a well know result due to Hutchinson \cite{Hut}, a unique non-empty compact set $X\subseteq\mathbb{R}^d$
satisfying 
\[
X=\bigcup_{i\in \A}\phi_i(X).
\]
This set $X$ is called the attractor of $\Phi$ and $\Phi$ is
commonly referred to as an iterated function system or IFS for short. When each element of the IFS
is an affine map we refer to the attractor as a self-affine set.
Similarly, when each element of the IFS is a similarity, i.e.\ there
exists $r_i\in (0,1)$ such that $|\phi_i(x)-\phi_i(y)|=r_i|x-y|$ for all $x,y\in \R^d$, we say that
the attractor is a self-similar set. For a self-similar IFS $\Phi$ we define the similarity dimension $\dim_{S}(\Phi)$ to
be the unique solution to $\sum_{i\in \A}r_i^s=1$. 
The similarity dimension is always an upper bound for the Hausdorff dimension of $X$. For self-affine
sets there is a similar upper bound for the Hausdorff dimension defined in terms of the affinity
dimension; see \cite{Fal} for its definition. The additional structure of affine maps or similarities makes
questions on the attractor more tractable and the two classes are the most
studied types of attractor.

A classical problem from
fractal geometry is to determine the metric and topological properties of self-similar sets and
self-affine sets; see \cite{Fal1,Fal2}. To make progress with this problem one often studies the
pushforwards of dynamically interesting measures onto the attractor. This approach has resulted in
many significant breakthroughs; see e.g.\ \cite{BaHoRa, Hochman, Hochman2, JoPoSi,PolSimon, Shm2,
Sol} and the references therein. Many important conjectures in this area can be summarised by the
statement: either an IFS contains an exact overlap,\footnote{We say that $\Phi$ contains an exact overlap
if there exists two words $(a_1,\ldots,a_n)$ and $(b_1,\ldots,b_m)$ such that $\phi_{a_1}\circ
\cdots \circ \phi_{a_n} = \phi_{b_1}\circ \cdots \circ \phi_{b_m}$} or the corresponding attractor
and the dynamically interesting measures supported upon it exhibit the expected behaviour. These
conjectures have been verified in certain special cases, see \cite{Hochman, Shm2, Varju2}. Part of
the motivation behind this paper is to obtain a deeper classification of iterated function systems
that goes beyond the exact overlap versus no exact overlap dichotomy.

\subsection{Diophantine approximation}
Given a function $\Psi:\mathbb{N}\to[0,\infty),$ we can define a limsup set in terms of
neighbourhoods of the rationals.
Let

\[
J(\Psi):=\Big\{x\in\mathbb{R}:\Big|x-\frac{p}{q}\Big|\leq \Psi(q) \textrm{ for i.m. } (p,q)\in
\mathbb{Z}\times\mathbb{N}\Big\}.
\]
The Borel-Cantelli lemma implies
that if $\sum_{q=1}^{\infty}q\cdot \Psi(q)<\infty,$ then $J(\Psi)$ has zero Lebesgue measure.
Interestingly, a theorem due to Khintchine shows that a partial converse to this statement holds.

\begin{theorem}[\cite{Khit}]
	\label{Khintchine}
	If $\Psi:\mathbb{N}\to [0,\infty)$ is decreasing and 
	\[
	\sum_{q=1}^{\infty}q\cdot
	\Psi(q)=\infty,
	\]
	then Lebesgue almost every $x\in \mathbb{R}$ is contained in $J(\Psi)$.
\end{theorem}
An example due to Duffin and Schaeffer shows that one cannot remove the monotonicity
assumptions from Theorem \ref{Khintchine} \cite{DufSch}. This lead to the famous Duffin and
Schaeffer conjecture that was recently proved by Koukoulopoulos and Maynard \cite{KouMay}.

Results analogous to Khintchine's theorem which show that the measure of a limsup set is determined
by the convergence/divergence of some naturally occurring volume sum are present throughout
Diophantine approximation and metric number theory (see \cite{BDV}). For our purposes, the important
aspect of the above is that by studying the metric properties of the sets $J(\Psi)$ for those $\Psi$
satisfying $\sum_{q=1}^{\infty}q\cdot \Psi(q)=\infty,$ one obtains a quantitative description of how
the rational numbers are distributed within the real numbers. In particular, the example due to
Duffin and Schaeffer of a $\Psi$ for which $\sum_{q=1}^{\infty}q\cdot \Psi(q)=\infty,$ yet $J(\Psi)$
has zero Lebesgue measure, reveals certain subtleties in the geometry of the rational numbers.

\subsection{Overlapping iterated function systems from the perspective of metric number
theory}Khintchine's theorem provides a quantitative description of how the rationals are distributed
within $\R$. The motivation behind the work discussed below comes from a desire to obtain an
analogous quantitative description for how an iterated function system overlaps.

Given a finite set $\A$ we let $\A^*=\bigcup_{n=1}^{\infty}\A^n$ denote the corresponding set of
finite words. Given an IFS $\{\phi_i\}_{i\in\A}$ and a word $\a=(a_1 \ldots a_n)\in \A^*,$ we let
$\phi_{\a}:=\phi_{a_1}\circ \cdots \circ\phi_{a_n}$. We also let $|\a|$ denote the length of a word
$\a.$ Now suppose we have an IFS $\Phi$, a function $\Psi:\A^*\to[0,\infty)$, and $z\in X$, we
define the following analogue of the set $J(\Psi)$:  
\[
W_{\Phi}(z,\Psi):=\left\{x\in \mathbb{R}^d:|x-\phi_{\a}(z)|\leq
\Psi(\a)\textrm{ for i.m.\ }\a\in \A^*\right\}.
\]
For the set $W_{\Phi}(z,\Psi)$ the role of the rationals is played by the images of $z$ obtained by
repeatedly applying elements of the IFS. Proceeding via analogy with Theorem \ref{Khintchine}, 
it is reasonable to expect
that there exists a divergence condition on volume sums which implies that being
contained in $W_{\Phi}(z,\Psi)$ holds almost surely with respect to some measure. One particular
instance of this could be formalised as follows: Let $\mathcal{H}^s$ be the $s$-dimensional
Hausdorff measure. Is it true that 
\begin{equation}
\label{divergence}
\sum_{n=1}^{\infty}\sum_{\a\in \A^n}\Psi(\a)^{\dim_{H}(X)}=\infty\Rightarrow  \mathcal{H}^{\dim_{H}(X)}(W_{\Phi}(z,\Psi))=\mathcal{H}^{\dim_{H}(X)}(X)?
\end{equation} 
The existence of a general class of $\Psi$ for which \eqref{divergence} holds demonstrates how well the
images of $z$ are spread out within $\R^d$.
Studying those $\Psi$ for which \eqref{divergence} holds provides a quantitative description of
how an IFS overlaps.

In a series of recent papers, the first author established that for many IFSs we do observe
Khintchine like behaviour, i.e.~\eqref{divergence} holds for some suitable class of $\Psi,$ see
\cite{Bak2,Bak,Bakerapprox2,BakOver}.  Related results had appeared previously in papers of Persson
and Reeve
\cite{PerRev,PerRevA}, and Levesley, Salp, and Velani \cite{LSV}. In
\cite{Bak2} it was shown that whenever $\Phi$ is an IFS consisting of similarities and satisfies the
open set condition, then an appropriate analogue of Theorem \ref{Khintchine} holds\footnote{This
	result in fact holds whenever $\Phi$ consists of conformal maps.}. See \cite{ABa} for some further related
work. The more
challenging and interesting case is when the underlying IFS satisfies $\dim_{S}(\Phi)> d$, or the
equivalent inequality for the affinity dimension. Loosely speaking, when these inequalities are
satisfied it is possible for a better rate of approximation to hold generically. It was shown in
\cite{BakOver} that for many parameterised families of IFSs for which $\dim_{S}(\Phi)> d$ holds for
each member of the family, or the equivalent inequality for affinity dimension, an
analogue of Khintchine's theorem holds generically. To detail this analogue we introduce the
following notation which we will use throughout.

Given a set $B\subset \mathbb{N},$ we define the 
the upper density of $B$ to be
\[
\overline{d}(B):=\limsup_{n\to\infty}\frac{\#\{1\leq j\leq n:j\in B\}}{n}.
\]
Given $\eps>0,$ let
\[
G_{\eps}:=\left\{g:\mathbb{N}\to[0,\infty):\sum_{n\in B}g(n)=\infty\,, \forall B\subseteq
\mathbb{N} \textrm{ s.t.\ } \overline{d}(B)>1-\eps\right\}.
\]
We also define
\begin{equation}
\label{H functions} G:=\bigcup_{\eps\in(0,1)}G_{\eps}.
\end{equation} 
For example, it can be shown that the function $g(n)=1/n$ is contained in $G$.

Let $\lambda\in(0,1)$ and $O$ be a $d\times d$ orthogonal matrix. For any
$\t=(t_1,\ldots,t_{\#\A})\in\R^{\#\A\cdot d}$ we can define an iterated function system
$\Phi_{\t}:=\{\phi_i(x)=\lambda\cdot Ox+ t_i\}_{i\in \A}$. We let $X_\t$ denote the corresponding attractor of
$\Phi_\t$. The following theorem follows from \cite[Theorem 2.6, Corollary 2.7]{BakOver} and demonstrates that Khintchine like behaviour typically occurs within this family.

\begin{theorem}
	\label{Exampletheorem}
	Fix $\lambda\in(0,1/2)$ and $O$ a $d\times d$ orthogonal matrix. Suppose $\frac{\log \#\A}{-\log \lambda}>d.$ Then for Lebesgue almost every
	$t\in\mathbb{R}^{\#\A\cdot d}$, for any $g\in G$ and $z\in X_{\t},$ the set
	\[
	\left\{x\in\R^d:|x-\phi_{\a}(z)|\leq \left(\frac{g(|\a|)}{(\#\A)^{|\a|}}\right)^{1/d} \textrm{ for
		i.m.\ } \a\in \A^*\right\}
	\]
	has positive Lebesgue measure.
\end{theorem}
Suitable analogues of Theorem \ref{Exampletheorem} hold with different rates of contraction and with
similarities replaced by affine maps (see \cite[Theorem 2.6]{BakOver}). 

The utility of studying IFSs using ideas from Diophantine approximation is emphasised by an observation made in \cite{BakOver}. Consider the parameterised family of IFSs given
by 
\[
\Phi_{t}=\Big\{\phi_1(x)=\frac{x}{2},\, \phi_{2}(x)=\frac{x+1}{2},\,
\phi_{3}(x)=\frac{x+t}{2},\, \phi_{4}(x)=\frac{x+1+t}{2}\Big\}.
\]
Here $t\in [0,1]$ and the
attractor of $\Phi_t$ is $[0,1+t]$. For $t\in [0,1]$ and a probability vector $\p=(p_1,p_2,p_3,p_4)$ we let $\mu_{\p,t}$
be the self-similar measure corresponding to $\p$ and $\Phi_t$ (see \cite{Fal2} for the definition of a self-similar measure). It was shown in \cite{BakOver} that
there exists $t,t'\in [0,1]$ such that for any probability vector $\p$ we have $\dim \mu_{\p,t}=\dim
\mu_{\p,t'}=\min\{\frac{-\sum_{i=1}^4p_i\log p_i}{\log 2},1\},$ and the set of $\p$ for which it is known that $\mu_{\p,t}$ is absolutely continuous equals
the set of $\p$ for which $\mu_{\p,t'}$ is known to be absolutely continuous. However there exists
$\Psi$ for which $W_{\Phi_{t}}(z,\Psi)$ has full measure within $[0,1+t]$ for all $z\in [0,1+t]$ and
$W_{\Phi_{t'}}(z,\Psi)$ has zero measure for all $z\in [0,1+t']$. In other words $\Phi_t$ and
$\Phi_{t'}$ are indistinguishable in terms of the properties of their self-similar measures, but
their overlapping behaviours can be distinguished using the language of the sets $W_{\Phi}(z,\Psi)$.

\subsection{Random models for iterated function systems}
The main results of \cite{BakOver} hold for several families of parameterised IFSs. In
the absence of a general result for parameterised families of IFSs, it is natural to study suitable
random analogues that mirror the key properties a family exhibits.
This approach benefits from small random perturbations that ``smooth out'' the parts that are
intractable in a deterministic approach.
This was employed in \cite{JoPoSi} by adding random translations to the deterministic linear parts
to determine the almost sure dimensions of their random attractors, as well as finding conditions for
absolute continuity. A complementary approach was taken in \cite{PSS} which randomised the linear
part while keeping the translates fixed. It further assumed that the linear parts were similarities
and that the randomisation is uniform for all cylinders in that level of the construction (knows as random homogeneous or
$1$-variable attractor, see \cite{tromodels}). 
A similar model was considered in \cite{Jarvenpaa17}, where the authors
determined the dimensions of random self-affine sets.
In this paper we randomise the linear part at every stage using the random recursive model, where we
allow the linear parts to be both
self-similar and self-affine. Theorem \ref{Main theorem} is the main
result of this paper. It gives sufficient conditions for a random model to ensure that an analogue
of Khintchine's theorem holds almost surely.\\

\noindent\textbf{Notation.} For two real valued functions $f$ and $g$ defined on some set $S$, we
write $f\ll g$ or $f=\O(g)$ if there exists $C>0$ such that $|f(x)|\leq C\cdot
g(x)$ for all $x\in S$. We write $f\asymp g$ if $f\ll g$ and $g\ll f$. 

Let $\A$ be a finite set and $i\in \A$. Given a finite word $\a\in \A^{*}$ we let
$|\a|_i:=\#\{1\leq k\leq |\a|:a_k=i\}$ denote the number of occurrences of the digit $i$ in $\a$.
Moreover, given two words $\a,\b\in \A^{*}$ we let $|\a\wedge \b|$ denote the maximal common prefix
of $\a$ and $\b$, assuming such a prefix exists. If no such prefix exists, $|\a\wedge \b|$ is the empty word.

\section{Our random model and statements of results} 
\label{sect:Results}
In this paper we will consider the stochastically self-similar (and self-affine) model which is is also known as the random recursive or
$\infty$-variable model. It is one of the most important models of randomness in fractal geometry
and was introduced, independently, by Falconer \cite{Falconer86} and Graf \cite{Graf87} and has
subsequently attracted a lot of attention, see e.g.\ \cite{Olsen, ShmerkinSuomala, tromodels} and
the references therein. To define this randomisation rigorously, we first define random functions $f_{\a}$ indexed by
$\a\in\A^*$. Each $f_{\a}$ is chosen independently from all other $\b\neq\a$ following a
distribution that only depends on the last letter of $\a$. 

Let $M_{d}\subset \R^{d^{2}}$ denote the set
of invertible $d\times d$ matrices with real entries satisfying $\|A\|<1$, where $\|.\|$ denotes the
usual operator norm. We write $S_d$ for those elements of $M_d$
that are also similarities, i.e.\ are a scalar multiple of an orthogonal matrix.
For each $i\in \A$ we let $\Omega_i$ be a subset of $M_{d}$ with operator norm uniformly bounded away from $1$. Moreover, for each $i\in A$ we let $\eta_i$ denote a Borel probability measure supported on $\Omega_i$.
We define a product measure indexed by the elements of $\A^*$ such that the distribution $\eta|_{\a}$ restricted to
$\a\in\A^*$ depends only on
the last letter $l(\a):=a_{|\a|}$. That is, we set
$\eta=\prod_{\a\in\A^*} \eta_{l(\a)}$ as the product measure on the product space
$\Omega =\prod_{\a\in\A^*}\Omega_{l(\a)}$. 
Thus, a particular realisation $\omega\in\Omega$ is a collection of matrices in $M_d$ indexed by
$\a\in\A^*$, where each entry is distributed according to its respective $\eta_{l(\a)}$.

We will make the distinction between $\omega\in\Omega$ as a realisation chosen with law $\eta,$ and
the linear component it defines at
a particular index by writing $A_{\omega,\a}(x): = (\omega)_{\a}\cdot x$ for the linear function
given by the random matrix indexed
by $\a$. We will often write $A_{\a}$ for $A_{\omega,\a}$ when the choice of $\omega$ is implicit.
Note that $A_{\a}$ is distributed with law $\eta_{l(\a)}$, the distribution corresponding
to the last letter of $\a$.
By definition, this function is independent from $A_{\b}$ for all $\b\in\A^*$ with 
$\b\neq\a$.

Let $\{t_i\}_{i\in \A}$ be a finite collection of distinct
translation vectors in $\R^d$, that is $i\neq j \Rightarrow |t_i-t_j|\neq 0$. For every $\omega\in\Omega$
we define a random contraction $f_{\omega,\a}$ for
every finite word $\a\in\A^*$ to be
\[
f_{\omega,\a}(x):= A_{\omega,\a}(x) + t_{l(\a)}.
\]
We emphasise that although the distribution of $A_{\omega,\a}$ only depends on the last letter of
$\a$, the exact realisation depends upon $\a$ and is independent of all
$\b\neq \a$. We will often omit the realisation $\omega$ from $f_{\omega,\a}$ when it is clear from context. Given $\omega\in \Omega$ and $(a_1 \ldots a_n)\in \A^*$ we denote the corresponding concatenation of matrices as follows:
$$\widehat{A}_{\omega,a_1\ldots a_n}:=A_{\omega,a_1}\circ \cdots \circ A_{\omega,a_1\ldots a_n}.$$
For a finite word $\a=(a_1,\ldots,a_n)$ and $\omega\in \Omega$ we let 
\[
\phi_{\omega,\a}(x):=f_{\omega,a_1}\circ \cdots \circ f_{\omega,a_1\ldots a_n}(x)
\]
Given $\omega\in \Omega$ we define the projection map  $\Pi_{\omega}:\A^{\mathbb{N}}\to\R^d$ via the
equation 
\[\Pi_{\omega}(\a) = \lim_{k\to
	\infty}\phi_{\omega,a_1\ldots a_k}(\mathbf{0}).
\]
Notice that $\mathbf{0}$ can be replaced with any element of $\mathbb{R}^d$. In addition, given a finite word $\a\in \A^*$ and
$\omega\in \Omega,$ we define the projection map  $\Pi_{\omega,\a}:\A^{\mathbb{N}}\to\R^d$ to be
\[
\Pi_{\omega,\a}(\b)=\lim_{k\to\infty} f_{\omega,\a b_1}\circ \cdots \circ f_{\omega,\a
	b_{1}\ldots b_k}(\mathbf{0}).
\]
Notice that for
any $\a\in \A^*$ and $\b\in \A^{\N}$ we have $$\Pi_{\omega}(\a \b)=\phi_{\omega,\a}(\Pi_{\omega,\a}(\b)).$$ In what follows we refer to the tuple $(\{\Omega_i\}_{i\in
	\A},\{\eta_i\}_{i\in \A},\{t_i\}_{i\in \A})$ as a random iterated function system or RIFS for short. For any $\omega\in \Omega$ its unique random attractor is defined to be
\[
F_\omega: = \bigcup_{\a\in\A^{\mathbb{N}}}\Pi_{\omega}(\a).
\]
$F_{\omega}$ is a non-empty compact set for all $\omega\in \Omega$. 
By definition, the set $F_{\omega}$ is stochastically self-similar in the sense that 
\[
  F_{\omega} \equiv_d\bigcup_{i\in\A}f_{\kappa,i}(F_{\tau_i})
\]
holds in distribution, where $\omega,\kappa,\tau_1,\dots,\tau_{\#\A}$ are independently realisations
in $(\Omega,\eta)$.

Given $\Psi:\A^*\to[0,\infty),$ $\b\in\A^{\mathbb{N}}$, and $\omega\in \Omega$, our random analogue
of the deterministic set $W_{\Phi}(z,\Psi)$ is defined to be 
\[
W_{\omega}(\b,\Psi):=\left\{x\in
\mathbb{R}^d:|x-\Pi_{\omega}(\a\b)|\leq \Psi(\a) \textrm{ for i.m.\ }\a\in \A^*\right\}.
\]

\subsection{An auxiliary family of sets}

Directly studying the sets $W_{\omega}(\b,\Psi)$ for a general $\Psi$
is a challenging problem. Instead, we study properties of an auxiliary family that we can then use
to deduce results about general $W_{\omega}(\b,\Psi)$. This auxiliary
family is defined below using dynamically interesting measures on $\A^{\N}$. As such it is necessary to introduce some definitions describing important properties of these measures.

The cylinder set associated with a finite word $\a=a_1\dots a_n\in \A^{*}$ is
\[
[\a]:=\{\b\in\A^\N:b_k=a_k \text{ for all } 1\leq k\leq n\}.
\]
Let $\sigma:\A^{\N}\to\A^{\N}$, $\sigma(a_1a_2\dots)=a_2a_3\dots$ denote the
usual left shift map. Given a probability measure $\m$ supported on $\A^{\mathbb{N}},$ we say that
$\m$ is $\sigma$-invariant if $\m([\a])=\m(\sigma^{-1}([\a]))$ for all finite words $\a\in \A^*$. We
call a probability measure $\m$ ergodic if $\sigma^{-1}(A)=A$ implies $\m(A)=0$ or $\m(A)=1$. Given
a $\sigma$-invariant, ergodic probability measure $\m$, we define the measure theoretic entropy of
$\m$ to be 
\[
h(\m):=\lim_{k\to\infty}\frac{-\sum_{\a\in \A^k}\m([\a])\log \m([\a])}{k}.
\]
Note that this limit always exists. We say that a probability measure $\m$ is slowly decaying if 
\[
c_{\m}:={\textrm{ess}\inf}\inf_{k\in\mathbb{N}}\frac{\m([a_1,\ldots,a_{k+1}])}{\m([a_1,\ldots,a_k])}>0.
\]
If $\m$ is slowly
decaying, then clearly for $\m$-almost every $\a\in \A^{\mathbb{N}}$ we have
\[
\frac{\m([a_1,\ldots,a_{k+1}])}{\m([a_1,\ldots,a_k])}\geq c_{\m}
\]
for all $k\in\mathbb{N}$.
Specific examples of slowly decaying measures include Bernoulli measures, and Gibbs measures for
H\"older continuous potentials (see \cite{Bow}). If $\m$ is a slowly decaying probability measure with $c_\m$ defined as above, then for each
$n\in\mathbb{N}$ we define the \emph{level set}
\begin{equation}\label{eq:levelDef}
L_{\m,n}:=\{\a\in\A^*: \m([a_1,\ldots,a_{|\a|}])\leq
c_{\m}^n<\m([a_1,\ldots,a_{|\a|-1}])\}.
\end{equation}
The elements of $L_{\m,n}$ are disjoint and the union of
their cylinders has full $\m$ measure. It follows from the slowly decaying property that cylinders
corresponding to elements of $L_{\m,n}$ have comparable measure up to a multiplicative constant.
Note that when $\m$ is the uniform $(\frac{1}{\#\A},\ldots,\frac{1}{\#\A})$-Bernoulli measure the
set $L_{\m,n}$ is simply $\A^n$.

Given $\b\in\A^{\mathbb{N}},$ a slowly decaying probability measure $\m,$ $\omega\in \Omega,$ and
$g:\mathbb{N}\to[0,\infty),$ we let
\[
U_{\omega}(\b,\m,g):=\left\{x\in \mathbb{R}^d:|x-\Pi_{\omega}(\a\b)|\leq (\m([\a])g(n))^{1/d}
\textrm{ for some } \a\in L_{\m,n}\text{ for i.m.\ }n\right\}.
\]
The sets $U_{\omega}(\b,\m,g)$ are the auxiliary sets that will allow us to deduce metric statements about certain
$W_{\omega}(\b,\Psi)$ for particular choices of $\Psi$ (see Corollary \ref{Bernoulli corollary}
below). The property of those $\Psi$ that allows us to use the sets $U_{\omega}(\b,\m,g)$ is
described in the following definition. Given a slowly decaying probability measure $\m$ and
$g:\mathbb{N}\to[0,\infty),$ we say that a function $\Psi$ is equivalent to $(\m,g)$ if
\[
\Psi(\a)\asymp(\m([\a])g(n))^{1/d}
\]
for all $\a\in L_{\m,n}$. 

If $\sum_{n=1}^{\infty}g(n)<\infty$ then it can be shown that $U_{\omega}(\b,\m,g)$ has zero Lebesgue measure for any choice of $\b,\m,$ and $\omega$. As such, to prove a Khintchine type theorem it is necessary to include a divergence assumption for the function $g$. In our results, the divergence assumption will be that $g$ is an element of $G$ (see \eqref{H functions}). 
\subsection{Statement of results}
To state our main result we require the following definitions. 
\begin{definition}
	We say that our random iterated function system is \textbf{\textit{non-singular}} if there exists $C>0$
	such that for all $i\in \A$, $x\in \cup_{\omega\in \Omega} \Pi_{\omega}(\A^{\mathbb{N}})$,
	and balls $B(y,r),$ we have 
	\[
	\eta_i(A\in\Omega_i\;:\;A\cdot x \in B(y,r) ) \leq C\cdot r^d.
	\]
\end{definition}
\begin{definition}
	We say that a random iterated function system is \textbf{\textit{distantly non-singular}} if there exists $C>0$
	such that for all $i\in \A$, $x\in \bigcup_{\omega\in \Omega} \Pi_{\omega}(\A^{\mathbb{N}})$, and $y\in \mathbb{R}^d\setminus B(0,\frac{\min_{i\neq j}|t_i-t_j|}{8})$, we have
	\[
	\eta_i(A\in\Omega_i\;:\;A\cdot x \in B(y,r) ) \leq C\cdot r^d.
	\]
\end{definition}
Note that the distantly non-singular condition only considers balls that are not ``too
near'' the origin, whereas the non-singular condition considers any ball.
Thus, being distantly non-singular is a weaker condition than being non-singular. We will use the distantly non-singular condition when dealing with similarities, and the
non-singular condition when dealing with affinities.
In the latter case we
will use an equivalent but more natural definition, defined in terms of ellipses.
\begin{definition}
	A random iterated function system is \textbf{\textit{non-singular}} if there exists $C>0$ such that for
	all $i\in\A$, $x\in \cup_{\omega\in\Omega}\Pi_\omega(\A^{\mathbb{N}})$, and ellipse $E\subset
	\R^d$, we have 
	\[
	\eta_i(A\in\Omega_i\;:\; A\cdot x \in E) \leq C \cdot \mathrm{Vol}(E).
	\]
\end{definition}
The equivalence of these two definitions can easily be seen by considering the special ellipse of the form $B(y,r)$ and by
covering the ellipse by small balls.

Given a RIFS $(\{\Omega_i\}_{i\in \A},\{\eta_i\}_{i\in \A},\{t_i\}_{i\in \A})$ and a
probability measure $\m$ on $\A^{\N}$ we associate the quantities
\[
\lambda'(\eta_i):=-\int_{\Omega_i}\log (|\Det(A)|)\,d\eta_i(A)
\]
and
\[
\lambda(\eta,\m):=\sum_{i\in \A}\m([i])\cdot \lambda'(\eta_i).
\]
We call $ \lambda(\eta,\m)$ the Lyapunov exponent of our RIFS with respect to $\m$.
We will make the running assumption throughout this paper that $\lambda'(\eta_i)\in\R$ and that the
logarithmic moment condition
\begin{equation}\label{eq:CramerAss}
  \log\int_{\Omega_i}\exp\left(s\log|\Det(A)|\right)\,d\eta_i(A)
  =\log\int_{\Omega_i}|\Det(A)|^s d\eta_i(A)<\infty
\end{equation}
is satisfied for all $i\in\A$ and $s\in\R$ with $|s|$ sufficiently small.
This assumption is made solely for the purpose of using Cram\'er's theorem on large deviations in
the proof of Theorem \ref{Main theorem}, and other suitable generalisations may be made. In
particular, this assumption is trivially satisfied if there exists $c>0$ such that $|\Det(A)|\geq
c>0$. We also note that the moment condition directly implies $\lambda'(\eta_i)\in\R$.

\clearpage
We are now in a position to state our main result.
\begin{theorem}
	\label{Main theorem}
	Let $(\{\Omega_i\}_{i\in \A},\{\eta_i\}_{i\in \A},\{t_i\}_{i\in \A})$ be a RIFS and assume one of
	the following:
	\begin{itemize}
		\item[A.]
		Assume $\Omega_i\subset S_d$ for all $i\in \A$ and that the RIFS is distantly non-singular.
		\item[B.]
		Assume $\Omega_i\subset M_d$ for all $i\in\A$ and that the RIFS is non-singular.
	\end{itemize}
	Suppose $\m$ is a slowly decaying $\sigma$-invariant ergodic probability measure such that $\frac{h(\m)}{\lambda(\eta,\m)}>1$. Then the following statements hold:
	\begin{enumerate}
		\item For any $\b\in \A^{\mathbb{N}},$ for
		$\eta$-almost every $\omega\in \Omega,$ for any $g\in G$ the set $U_{\omega}(\b,\m,g)$ has
		positive Lebesgue measure.
		\item  For any $\b\in \A^{\mathbb{N}},$ for
		$\eta$-almost every $\omega\in \Omega,$ for any $\Psi:\A^{*}\to[0,\infty)$ the set $W_{\omega}(\b,\Psi)$ has positive Lebesgue measure if there exists $g\in G$ such that $\Psi$ is
		equivalent to $(\m,g)$.
	\end{enumerate}
\end{theorem}
When restricting to Bernoulli probability measures, the second statement from Theorem \ref{Main
	theorem} implies the following corollary.
\begin{corollary}
	\label{Bernoulli corollary}
	Let $(\{\Omega_i\}_{i\in \A},\{\eta_i\}_{i\in \A},\{t_i\}_{i\in \A})$ be an RIFS and assume one of
	the following.
	\begin{itemize}
		\item[A.] Assume $\Omega_i\subset S_d$ for all $i\in \A$ and that the RIFS is distantly
		non-singular. 
		\item[B.] Assume $\Omega_i\subset M_d$ for all $i\in \A$ and that the RIFS is 
		non-singular. 
	\end{itemize}
	Let $(p_i)_{i\in \A}$ be a probability vector satisfying
	$\frac{-\sum_{i\in \A} p_i\log p_i}{ \sum_{i\in \A}p_i\cdot \lambda'(\eta_i)}>1$. Then for any
	$\b\in \A^{\mathbb{N}},$ for $\eta$-almost every $\omega\in \Omega,$ the set 
	\[
	\left\{x\in
	\R^d:|x-\Pi_{\omega}(\a\b)|\leq \left(\frac{\prod_{k=1}^{|\a|}p_{a_k}}{|\a|}\right)^{1/d}\textrm{ for i.m.
	}\a\in \A^*\right\}
	\]
	has positive Lebesgue measure. 
\end{corollary}

By the compactness of $F_{\omega}$ it follows that $U_{\omega}(\b,\m,g)$ is a subset of $F_{\omega}$
whenever $g$ is bounded and $\m$ is non-atomic. Therefore Theorem \ref{Main theorem} immediately implies
the following result which can be seen to generalise the work of Peres,
Simon, and Solomyak \cite{PSS} for $1$-variable RIFS in $\R$ and the work of Koivusalo \cite{Koivusalo13}.
\begin{corollary}
	\label{measure corollary}
	Let $(\{\Omega_i\}_{i\in \A},\{\eta_i\}_{i\in \A},\{t_i\}_{i\in \A})$ be a RIFS and assume one of the following:
	\begin{itemize}
		\item[A.]  Assume $\Omega_i\subset S_d$ for all $i\in \A$ and that the RIFS is distantly 
		non-singular.
		\item[B.]  Assume $\Omega_i\subset M_d$ for all $i\in \A$ and that the RIFS is 
		non-singular.
	\end{itemize}
	If there exists a slowly decaying $\sigma$-invariant ergodic probability measure $\m$
	satisfying $\frac{h(\m)}{ \lambda(\eta, \m)}>1,$ then for $\eta$ almost every $\omega\in \Omega$
	the set $F_{\omega}$ has positive Lebesgue measure.
\end{corollary}

The rest of the paper is organised as follows. In Section \ref{technical} we prove several technical results that will enable us to prove Theorem \ref{Main theorem} in Section \ref{proof section}. In Section \ref{sec:Corollary} we demonstrate how Corollary \ref{Bernoulli corollary} follows from Theorem \ref{Main theorem}. In Section \ref{example section} we detail some examples of RIFSs that satisfy either assumption A or assumption B from the statement of our results. Finally in Section \ref{discussion section} we make some concluding remarks.

\section{Technical results}
\label{technical}
In this section we prove a number of technical results that will enable us to prove Theorem \ref{Main theorem}. In the first subsection we prove Proposition \ref{important properties prop}. This proposition allows us to assert that for $\eta$ almost every $\omega\in \Omega,$ for $n$ sufficiently large there exists a large subset $\widetilde{L}_{\m,n,\eps_{1}}(\omega)$ contained in $L_{\m,n}$ for which each element satisfies good determinant bounds and good measure decay bounds that are described in terms of a parameters $\eps_{1}>0$. In the second subsection we prove Lemma \ref{transversality lemma rework}. This lemma provides a good upper bound for the probability that two projections are close to each other. In the final subsection we recall some general results from \cite{BakOver} and \cite{BerVel2} which can be used to ensure that a limsup set has positive Lebesgue measure.

\subsection{Constructing $\widetilde{L}_{\m,n,\eps_{1}}$}
Given a RIFS and a slowly decaying $\sigma$-invariant ergodic probability measure $\m$, recall that the Lyapunov exponent of our RIFS with respect to $\m$ is
\[
\lambda(\eta,\m)= \sum_{i\in\A}\m([i])\lambda'(\eta_i)=-\sum_{i\in \A}\m([i])\int_{\Omega_i}\log |\Det(A)|d\eta_i(A),
\]
and the entropy of $\m$ is given by
\[
h(\m) = \lim_{n\to\infty}\frac{-\sum_{\a\in\A^n}\m([\a])\log\m([\a])}{n}.
\]
The Shannon-McMillan-Breiman theorem tells us that for $\m$-almost every $\a\in \A^{\N}$ we have 
\[
\lim_{n\to\infty}\frac{-\log \m([a_1,\ldots,a_n])}{n}=h(\m).
\] We will combine this statement with Egorov's theorem to obtain uniform estimates on the measures of cylinders. The first step in our proof of Proposition \ref{important properties prop} is the following proposition which states that for $\eta$ almost every $\omega$ there is a large subset of $\A^{\N}$ on which we have good determinant bounds. 

\begin{proposition}
	\label{tree prop}
Fix a RIFS and a $\sigma$-invariant ergodic probability measure $\m.$ Then for any $\eps_1>0,$ there exists $C=C(\m,\eta,\eps_1)>0$ such that for $\eta$ almost every
	$\omega\in\Omega,$ there exists $N=N(\omega)\in \mathbb{N}$ such that\footnote{The lower
	bound of $13/16$ is arbitrary and can be replaced by any value less than $1$.} 
	\[
	\m\left (\a\in
	\A^{\mathbb{N}}: |\Det(\widehat{A}_{\omega,a_1\ldots
		a_n})|\in\left(\frac{e^{-n(\lambda(\eta,\m)+\eps_1)}}{C},Ce^{-n(\lambda(\eta,\m)-\eps_1)}\right)\textrm{
		for all }n\geq N\right)>\frac{13}{16}.
	\]
\end{proposition}
\begin{proof}
We fix a RIFS, a $\sigma$-invariant ergodic probability measure $\m,$ and let $\epsilon_1>0$. Let $\eps_2=\eps_2(\eps_1)>0$ be sufficiently small such that 
	\begin{equation}
	\label{eps2}
	\eps_2\left( 1+\sum_{i\in\A}\lambda'(\eta_i) \right)<\eps_1.
	\end{equation}
	By an application of the Birkhoff Ergodic theorem and Egorov's theorem, there exists $C_1=C_1(\m,\eps_{2})>1$ such that if we let $$\Sigma_{\m}:=\left\{\a\in \A^\mathbb{N}:\frac{e^{n(\m([i])-\eps_2)}}{C_1}\leq e^{|(a_k)_{k=1}^{n}|_{i}}\leq C_1e^{n(\m([i])+\eps_2)}\, \textrm{ for all }i\in \A,\, n\in \N\right\}$$ then 
	\begin{equation}
	\label{measure bound}
	\m(\Sigma_{\m})> \frac{15}{16}.
	\end{equation} 
	For each $n\in \mathbb{N},$ let
	\[
	\Sigma_{\m,n}:=\{\a\in \A^{n}: [a_1 \ldots  a_n]\cap \Sigma_{\m}\neq \emptyset\}
	\]
	be the words of length $n$ with ``good'' digit frequencies.
	
	 We split the remainder of our proof into two parts. In the first part we obtain an exponential upper bound for the probability that for a specific $\a\in \Sigma_{\m,n}$ the determinant of $\wA_{\omega,\a}$ behaves poorly. In the second part we use this bound to show that for almost every $\omega\in \Omega,$ there exists a large subset of $\A^{\N}$ upon which the determinant behaves well.  \\

	\noindent \textbf{Part 1: $|\Det(\wA_{\omega,\a})|$ is regular with high probability.}
	Let us temporarily fix some element $\a=a_1\dots a_n\in \Sigma_{\m,n}$.
	We want to obtain a good upper bound for the
	probability that 
	\[
	|\Det(\wA_{\omega,a_1\ldots a_n})|\notin
	\left(\frac{e^{-n(\lambda(\eta,\m)+\eps_1)}}{C_2},C_2e^{-n(\lambda(\eta,\m)-\eps_1)}\right),
	\]
	for some $C_2>0$.
	Since the determinant is multiplicative, and $A_{\omega,a_1\dots a_i}$ is independent of $A_{\omega,a_1\dots a_j}$
	for $i\neq j$, we can break up the determinant into the $\#\A$ different contributions coming from each of the
	probability measures $\eta_i$. This means that there exists $\#\A$ words $\b_1,\ldots,\b_{\#\A}$
	consisting solely of the digits $1,\dots,\#\A$ respectively, such that 
	\begin{align}
	|\Det(\wA_{\omega,\a})|&=\prod_{i\in \A}|\Det(\wA_{\b_{i}})|\nonumber\\
	\intertext{and}
	|\b_i|_i=|\b_{i}|&=|(a_1 \ldots a_n)|_{i} \textrm{ for all }i\in \A.\nonumber
      \end{align}Moreover, each element of the word $\b_i$ is chosen independently each with respect to the
	probability measure $\eta_i$.
	Therefore it follows from Cram\'er's theorem on large deviations and our assumption
	\eqref{eq:CramerAss}, that for each $i\in \A$ there exists
	$\rho_i=\rho_i(\eps_1,\eta_i)\in(0,1)$ and $C_3=C_3(\eps_1,\eta_i)>0$ such that
	\[
	\eta\left(\omega:|\Det(\wA_{\b_i})|\notin\left(e^{-|\b_i|(\lambda'(\eta_i)+\eps_2)},e^{-|\b_i|(\lambda'(\eta_i)
		-\eps_2)}\right)\right)\leq C_3\rho_i^{|\b_i|}.
	\]
	Given the finiteness of $\A$ and the fact that $\a\in\Sigma_{\m,n}$, which implies a lower bound for $|\b_i|$ in terms of a constant times $n$, one can derive a uniform exponential bound in $n$. In particular, there exists
	$\rho=\rho(\eps_1, \eta,\m)\in(0,1)$ and $C_4=C_4(\eps_1,\eta,\m)$ such that 
	\[
	\eta\left(\omega:|\Det(\wA_{\b_i})|\notin \left(e^{-|\b_i|(\lambda'(\eta_i)+\eps_2)},e^{-|\b_i|(\lambda'(\eta_i)
		-\eps_2)}\right)\right)\leq C_4 \rho^{n} \textrm{ for all }i\in \A.
	\]
	For each $\a\in \Sigma_{\m,n}$ consider the event 
	\[
	E_{\a}=\left\{\omega\in\Omega\;:\; |\Det(\wA_{\omega,\a})| = \prod_{i\in \A}|\Det(\wA_{\b_i})|\notin
	\left(\prod_{i\in\A}e^{-|\b_i|(\lambda'(\eta_i)+\eps_2)},
	\prod_{i\in\A}e^{-|\b_i|(\lambda'(\eta_i)-\eps_2)}\right)\right\}.
	\]
	Clearly, if $\omega\in E_{\a}$ then $|\Det(\wA_{\b_i})|\notin
	\left(e^{-|\b_i|(\lambda'(\eta_i)+\eps_2)},e^{-|\b_i|(\lambda'(\eta_i)
		-\eps_2)}\right)$ for some $i$ and therefore
	\[
	\eta(E_{\a}) \leq \#A\cdot C_4\rho^n.
	\]
	Let $C_2 = C_2(\eta)>0$ be such that
	\begin{equation}\label{eq:constBound}
	C_2 \geq\max\left\{ e^{-\#\A\log C_1\sum_i\lambda'(\eta_i)}, e^{\#\A \log C_1\sum_{i}\lambda'(\eta_i)} \right\}.
	\end{equation}
	Manipulating the lower bound we obtain that for each $\a\in \Sigma_{\m,n}$ we have
	\begin{align*}
	\prod_{i\in\A}e^{-|\b_i|(\lambda'(\eta_i)+\eps_2)}
	&=e^{-n\eps_2}\prod_{i\in\A}e^{-|\a|_i\lambda'(\eta_i)}\\
	&\geq e^{-n\eps_2}\prod_{i\in\A} e^{-\lambda'(\eta_i)(n(\m([i])-\eps_2)-\log C_1)}\\
	&\geq e^{\#\A\log C_1 \cdot \sum_i\lambda'(\eta_i)}
	\exp\left( -n\left(\eps_2-\sum_{i\in\A}\lambda'(\eta_i)(\m([i])-\eps_2)\right) \right)\\
	&\geq C_2^{-1} \exp\left( -n\left( \eps_2+\lambda(\eta,\m)+\eps_2\sum_{i\in\A}\lambda'(\eta_i) \right)
	\right)&&\text{by \eqref{eq:constBound}}\\
	&\geq C_2^{-1} e^{-n(\lambda(\eta,\m)+\eps_1)}&&\text{by \eqref{eps2}}.
	\end{align*}
	The following upper bound for each $\a\in \Sigma_{\m,n}$ is proved similarly
	$$\prod_{i=1}^{\#\A}e^{-|\b_i|(\lambda'(\eta_i)-\eps_2)}C_2 \leq
	e^{-n(\lambda(\eta,\m)-\eps_1)}.$$
	We conclude that for each $\a\in \Sigma_{\m,n}$ the event 
	\[
	E_{\a}' = \left\{\omega\in\Omega\;:\; |\Det(\wA_{\omega,\a})| \notin \left( C_2^{-1}
	e^{-n(\lambda(\eta,\m)+\eps_1)}, C_2 e^{-n(\lambda(\eta,\m)-\eps_1)}  \right)  \right\}
	\]
	satisfies $E_{\a}'\subset E_{\a}$ and so $\eta(E_{\a}')\leq \eta(E_{\a}) \leq \#\A C_4 \rho^n$.
	In summary, we have shown that
	\begin{equation}\label{measure bound2}
	\eta\left(  \omega\in\Omega\;:\; |\Det(\wA_{\omega,\a})| \notin \left( C_2^{-1}
	e^{-n(\lambda(\eta,\m)+\eps_1)}, C_2 e^{-n(\lambda(\eta,\m)-\eps_1)}  \right) 
	\right)\leq\#\A C_4 \rho^n
	\end{equation}
	for all $\a\in\Sigma_{\m,n}$.\\

	\noindent \textbf{Part 2: Constructing a large subset of $\A^{N}$ on which the determinant is regular. }Let $\eps_3>0$ and $\theta\in(0,1)$ be such that 
	\begin{equation}\label{eq:exponentialdef}
	\rho\theta^{-1}<1\textrm{ and }\theta':=e^{2\eps_3}\theta<1.
	\end{equation}
	Combining the Shannon-McMillan-Breiman theorem, Egorov's theorem,  and \eqref{measure bound}, we may
	assert that there exists $C_{5}=C_{5}(\m,\eps_{3})>0$ such that if we let 
	\[
	\Sigma^{*}_{\m}:=\Sigma_{\m}\cap
	\left\{\a\in \A^{\N}:\frac{e^{-n(h(\m)+\eps_3)}}{C_5}\leq \m([a_1\ldots a_n])\leq
	C_{5}e^{-n(h(\m)-\eps_3)}\textrm{ for all }n\in\N\right\}
	\]
	be the set of sequences with ``good'' digit frequency and ``good'' measure decay, then 
	\begin{equation}
	\label{3/4 bound}
	\m(\Sigma_{\m}^{*})>\frac{14}{16}.
	\end{equation}
	Again we define the level sets by
	\[
	\Sigma_{\m,n}^{*}:=\{\a\in \A^n: [a_1\ldots a_n]\cap \Sigma_{\m}^{*}\neq \emptyset\}
	\]
	and we note that $\Sigma_{\m,n}^{*}\subseteq \Sigma_{\m,n}$ for all $n$.
	Therefore \eqref{measure bound2} also applies to elements of $\Sigma_{\m,n}^{*}$. 
	
	Using the measure bounds coming from the definition of
	$\Sigma_{m}^{*}$ we have the following upper bound for the
	cardinality of $\Sigma_{\m,n}^{*}$: 
	\begin{equation}
	\label{cardinality upper bound}
\#\Sigma_{\m,n}^{*}\leq C_{5}e^{n(h(\m)+\eps_3)}.
	\end{equation} 
	We can bound the expected number of words
	\[
	B_{\m,n}(\omega): = \left\{ \a\in\Sigma^*_{\m,n}\;:\;|\Det(\wA_{\omega,\a})|\notin\left(
	\frac{e^{-n(\lambda(\eta,\m)+\eps_1)}}{C_1},C_1 e^{-n(\lambda(\eta,\m)-\eps_1)} \right) \right\}
	\]
	that do
	not have good Lyapunov exponent using \eqref{measure bound2}:
	\begin{align*}
	\int_{\Omega}\#B_{\m,n}(\omega)\,d\eta =& \int_{\Omega}\#\left\{\a\in \Sigma_{\m,n}^*:|\Det(\wA_{\omega,\a})|\notin
	\left(\frac{e^{-n(\lambda(\eta,\m)+\eps_1)}}{C_2},C_2e^{-n(\lambda(\eta,\m)-\eps_1)} \right)\right\}d\eta\\
	=&\sum_{\a\in \Sigma_{\m,n}^*}\int_{\Omega}\chi\left(|\Det(\wA_{\omega,\a})|\notin
	\left(\frac{e^{-n(\lambda(\eta,\m)+\eps_1)}}{C_2},C_2e^{-n(\lambda(\eta,\m)-\eps_1)}
	\right)\right)d\eta\\
	\leq& C_{4}\#\A\sum_{\a\in \Sigma_{\m,n}^*} \rho^n\\
	\leq& C_{4}\#\A \cdot\rho^n\#\Sigma_{\m,n}^{*}.
	\end{align*} 
	By Markov's inequality, we have 
	\begin{align*}
	\eta\left(\omega:\#B_{\m,n}(\omega)\geq \#\Sigma_{\m,n}^*\cdot \theta^n \right)
	\leq &  C_{4}\#\A \cdot\rho^n\#\Sigma_{\m,n}^{*}
	\theta^{-n}(\#\Sigma_{\m,n}^*)^{-1}\\
	\leq &  C_{4}\#\A \rho^n\theta^{-n}.
	\end{align*}
	Therefore by \eqref{eq:exponentialdef}
	\[
	\sum_{n\in\N}\eta\left(\omega:\#B_{\m,n}(\omega)\geq \#\Sigma_{\m,n}^*\cdot \theta^n  \right)
	<\infty.
	\]
	It follows from the Borel-Cantelli Lemma that for $\eta$-almost every $\omega\in\Omega$ there
	exists $N=N(\omega)\in\N$ such that
	\begin{equation}\label{eq:decayingProbs}
	\#B_{\m,n}(\omega) \leq \#\Sigma_{\m,n}^* \cdot \theta^n
	\end{equation}
	for all $n\geq N$.
	It follows now from the definition of $\Sigma_{\m,n}^*$ that for $\eta$ almost every $\omega$, there
	exists $N\in \N$ such that for all $n\geq N$,
	\begin{align*}
	\m\left(\bigcup_{a_1\ldots a_n\in B_{\m,n}(\omega) }[a_1\ldots a_n]\right)&\leq C_{5}e^{-n(h(\m)-\eps_3))}\cdot \# B_{\m,n}(\omega)\\
	&\leq C_{5}e^{-n(h(\m)-\eps_3))}\cdot  \# \Sigma_{\m,n}^*\cdot\theta^n&&\text{by
		\eqref{eq:decayingProbs}}\\
	&\leq (C_{5})^{2}e^{-n(h(\m)-\eps_3))}e^{n(h(\m)+\eps_3)}\theta^n&& \eqref{cardinality upper bound}\\
	&\leq  (C_{5})^{2}e^{2n\eps_3}\theta^n\\
	&\leq  (C_{5})^{2}(\theta')^n&&\text{by \eqref{eq:exponentialdef}}.
	\end{align*}Replacing $N$ with some larger value if necessary, we may assume that 
	\[
	\sum_{n=N}^{\infty}\m\left(\bigcup_{a_1\ldots a_n\in B_{\m,n}(\omega) }[a_1\ldots a_n]\right)\leq
	\sum_{n=N}^{\infty}(C_{5})^2(\theta')^n<1/16
	\]
	holds for $\eta$-almost every $\omega\in\Omega$. Therefore, for $\eta$-almost every $\omega$ we have
	\[
	\m\left(\bigcup_{n=N}^{\infty}\bigcup_{a_1\ldots a_n\in B_{\m,n}(\omega) }[a_{1}\ldots a_{n}]\right)<\frac{1}{16}.
	\]
	Combining this inequality with \eqref{3/4 bound} we see that for $\eta$-almost
	every $\omega,$ for $N$ sufficiently large we have 
	\[
	\m\left(\Sigma_{\m}^{*}\setminus
	\bigcup_{n=N}^{\infty}\bigcup_{a_1\ldots a_n\in B_{\m,n}(\omega) }[a_{1}\ldots a_{n}]\right)>\frac{13}{16}.
	\]
	Finally, we observe that if
	\[
	\a\in \Sigma_{\m}^{*}\setminus \bigcup_{n=N}^{\infty}\bigcup_{a_1\ldots a_n\in
		B_{\m,n}(\omega) }[a_1,\ldots,a_n]
	\]
	then $\a$ satisfies
	\[
	|\Det(\wA_{\omega,a_1,\ldots,a_n})|\in \left(\frac{e^{-n(\lambda(\eta,\m)+\eps_1)}}{C_1},C_1e^{-n(\lambda(\eta,\m)-\eps_1)} \right)
	\]
	for all $n\geq N$. This completes our proof. 
\end{proof}

We now adapt Proposition \ref{tree prop} into a meaningful statement regarding the level sets $L_{\m,n}$.  Instead of dealing with $L_{\m,n}$ directly it is useful to restrict to the following large subset upon which we have strong measure decay estimates. For any slowly decaying $\sigma$-invariant ergodic probability measure $\m$ and $\eps_1>0$, we can use the Shannon-McMillan Breiman theorem and
Egorov's theorem to choose $C_2(\m,\eps_1)>0$ such that the set $L_{\m,n,\eps_1}\subseteq
L_{\m,n}$ defined as follows 
\begin{equation}
\label{uniform entropy condition}
L_{\m,n,\eps_1} := \left\{ \a\in L_{\m,n}\;:\;\frac{e^{-k(h(\m)+\eps_1)}}{C_2}\leq
\m([a_1\ldots a_k])\leq C_{2}e^{-k(h(\m)-\eps_1)} \text{ for all } 1\leq k \leq |\a|\right\}
\end{equation}
satisfies
\begin{equation}
\m\left(\bigcup_{\a\in L_{\m,n,\eps_1}}[\a]\right)>15/16.
\end{equation}

\begin{proposition}\label{important properties prop}
Fix a RIFS and a slowly decaying $\sigma$-invariant ergodic probability measure $\m.$ Then for any $\eps_1>0,$ there exists $C=C(\m,\eta,\eps_1)>0$ such that for almost every $\omega\in \Omega$, there exists
	$N_1=N_{1}(\omega)\in \N$ and $N_2=N_2(\omega)\in \N$ such that for all $n\geq N_2$ there exists
	$\widetilde{L}_{\m,n,\eps_1}(\omega)\subseteq L_{\m,n,\eps_1}$ satisfying:
	\begin{enumerate}
		\item For each $\a\in \widetilde L_{\m,n,\eps_1}(\omega)$ we have 
		\[
		|\Det(\wA_{\omega,\a})|\in \left(\frac{e^{-n(\lambda(\eta,\m)+\eps_1)}}{C},Ce^{-n(\lambda(\eta,\m)-\eps_1)} \right)
		\]
		for all $ N_1\leq n \leq |\a|$.
		\item $\# \widetilde{L}_{\m,n,\eps_1}\asymp c_{\m}^{-n}$ for all $n\geq N_2$.
	\end{enumerate}
\end{proposition}

\begin{proof}
	Let $\omega$ belong to the full measure set whose existence is asserted by Proposition
	\ref{tree prop}. Let $N_1=N_1(\omega)$ denote the large $N$ whose existence is also
	guaranteed by this proposition. Since $\m$ is non atomic, we may
	choose $N_2=N_{2}(\omega)$ sufficiently large such that for all $n\geq N_2$, each $\a\in
	L_{\m,n}$ satisfies $|\a|\geq N_1$. 
	
	By Proposition \ref{tree prop} we have $\m(H(\omega))>13/16$, where 
	\[
	H(\omega):=\left\{\a\in
	\A^{\mathbb{N}}: |\Det(\wA_{\omega,a_1\ldots
		a_n})|\in\left(\frac{e^{-n(\lambda(\eta,\m)+\eps_1)}}{C_1},C_1e^{-n(\lambda(\eta,\m)-\eps_1)}\right)\textrm{
		for all }n\geq N_{1}\right\}
	\]
	and $C_1>0$ is the constant guaranteed by Proposition \ref{tree prop}.
	
	For $n\geq N_2$ define 
	\[
	\widetilde{L}_{\m,n,\eps_1}(\omega):=\{\a\in L_{\m,n,\eps_1}:[\a]\cap H(\omega)\neq\emptyset\}.
	\]
	Notice that Property $1.$ is immediately satisfied by $\widetilde{L}_{\m,n,\eps_1}(\omega)$. To see that Property $2.$
	holds notice that $\m(H(\omega)\cap
	\cup_{\a\in L_{\m,n,\eps_1}}[a])>12/16$ follows from from the bounds $\m(\cup_{\a\in L_{\m,n,\eps_1}}[a])>15/16$ and $\m(H(\omega))>13/16$. Our cardinality bound now follows because $\m([\a])\asymp c_{\m}^n$ for each $\a\in \widetilde{L}_{\m,n,\eps_{1}}(\omega)$.
\end{proof}

\subsection{Transversality estimates}
To prove Theorem \ref{Main theorem} we need the following transversality lemma that bounds the probability that two points in the attractor are close. This is the only part in the proof where we use our non-singularity assumptions.

\begin{lemma}
	\label{transversality lemma rework}
	Let $(\{\Omega_i\}_{i\in \A},\{\eta_i\}_{i\in \A},\{t_i\}_{i\in \A})$ be a RIFS and assume one of
	the following:
	\begin{itemize}
		\item[A.] Assume that $\Omega_i\in S_d$ for all $i\in \A$ and that the RIFS is distantly
		non-singular;
		\item[B.] Assume that $\Omega_i\in M_d$ for all $i\in\A$ and the RIFS is 
		non-singular.
	\end{itemize}
	Let $\b\in \A^{\N}$ and $\a,\a'\in \A^{*}$ be two distinct words such that neither one is the
	prefix of the other. Then for any $C>0$ and $s>0$ and all $0<\eps<s$,
	\begin{align}&\int_{\Omega}\chi_{[-r,r]}(|\Pi_{\omega}(\a\b)-\Pi_{\omega}(\a'\b)|)\cdot
	\chi\left(\omega:\Det(\wA_{\omega,a_1\ldots a_n})\in \left(\frac{e^{-n(s+\eps)}}{C},Ce^{-n(s-\eps)}\right) \textrm{ for all
	}1\leq n\leq |\a| \right)\nonumber\\
	&\hspace{10em}\cdot \chi\left(\omega:\Det(\wA_{\omega,a_1'\ldots a_n'})\in
	\left(\frac{e^{-n(s+\eps)}}{C},Ce^{-n(s-\eps)}\right) \textrm{ for all }1\leq n\leq |\a'|
	\right)\,d\eta\nonumber\\
	&=\O(r^d\cdot C\cdot e^{|\a\wedge\a'|(s+\eps)}).\nonumber\\\label{eq:ineqproof}
	\end{align}
\end{lemma}
\begin{proof}
We split our proof into two parts.\\

	\noindent \textbf{Proof under assumption A.}
	First, assume that $|\a\wedge\a'|\geq 1$, i.e.\ that $\a$ and $\a'$ share a common prefix. Note that
	by assumption we also have $|\a\wedge\a'|<\min\{|\a|,|\a'|\}$. Let $\c$ and $\c'$ be the unique words such
	that $\a\b = (\a\wedge\a')\c\b$ and $\a'\b=(\a\wedge\a')\c'\b$.
	We emphasise that $\c\b$ and $\c'\b$ must have distinct first letter. We highlight the following inequality
	\begin{multline*}
	\chi\left(\omega\;:\;\Det(\wA_{\omega,a_1\ldots a_n})\in \left( \frac{e^{-n(s+\eps)}}{C}, C e^{-n(s-\eps)} \right),\;\forall
	1\leq n \leq |\a|\right)\\
	\cdot \chi\left(\omega\;:\;\Det(\wA_{\omega,a_1'\ldots a_n'})\in \left( \frac{e^{-n(s+\eps)}}{C}, C e^{-n(s-\eps)} \right),\;\forall
	1\leq n \leq |\a'|\right)\\
	\leq \chi\left( \omega\;:\;\Det(\wA_{\omega,\a\wedge\a'})\in\left( \frac{e^{-|\a\wedge\a'|(s+\eps)}}{C},
	Ce^{-|\a\wedge\a'|(s-\eps)} \right)\right)
	\end{multline*}
	This implies
	\begin{align}
	&\chi_{[-r,r]}(|\Pi_{\omega}(\a\b)-\Pi_{\omega}(\a'\b)|)\cdot
	\chi\left(\omega:\Det(\wA_{\omega,a_1\ldots a_n})\in \left(\frac{e^{-n(s+\eps)}}{C},Ce^{-n(s-\eps)}\right),\;\forall1\leq n\leq
	|\a| \right)\nonumber\\
	&\hspace{12em}\cdot \chi\left(\omega:\Det(\wA_{\omega,a_1'\ldots a_n'})\in
	\left(\frac{e^{-n(s+\eps)}}{C},Ce^{-n(s-\eps)}\right),\;\forall1\leq n\leq |\a'| \right)\nonumber\\
	&\leq \chi(\omega\,:\,\Pi_{\omega}(\a\b)-\Pi_{\omega}(\a'\b)\in B(0,r))\cdot
	\chi\left( \omega\;:\;\Det(\wA_{\omega,\a\wedge\a'})\in\left( \frac{e^{-|\a\wedge\a'|(s+\eps)}}{C},
	Ce^{-|\a\wedge\a'|(s-\eps)} \right)\right)\nonumber\\
	&=\chi(\omega\,:\,\Pi_{\omega,\a\wedge\a'}(\c\b)-\Pi_{\omega,\a\wedge\a'}(\c'\b)\in(\wA_{\omega,\a\wedge\a'})^{-1}(
	B(0,r)))\nonumber\\
	&\hspace{15em}\cdot
	\chi\left( \omega\;:\;\Det(\wA_{\omega,\a\wedge\a'})\in\left( \frac{e^{-|\a\wedge\a'|(s+\eps)}}{C},
	Ce^{-|\a\wedge\a'|(s-\eps)} \right)\right)\label{eq:lastline}.
	\end{align}
	Since $A_{\omega,\a\wedge\a'}$ is by assumption $A$ a similarity, its contraction rate is
	$\Det(A_{\omega,\a\wedge\a'})^{1/d}$. Therefore, \eqref{eq:lastline} can be bounded above
	by
	\begin{equation}
	\label{eq:tosplit}
	\chi_{1}:=\chi(\omega\,:\,\Pi_{\omega,\a\wedge\a'}(\c\b)-\Pi_{\omega,\a\wedge\a'}(\c'\b)\in
	B(0,r\cdot C^{1/d}e^{|\a\wedge\a'|(s+\eps)/d})).
	\end{equation}
	We remark that the iterative definition of the random maps give the identity
	\[
	\Pi_{\omega,\a\wedge\a'}(\c\b) = A_{\omega,\a\wedge\a' c_1}(\Pi_{\omega,\a\wedge\a'
		c_1}(\sigma(\c\b)))+t_{c_1}.
	\]
	Write $r^{*} = \min_{i\neq j}|t_i-t_j|$. 
	Since $\c\b$ and $\c'\b$ differ in their first letter we have
	$|t_{c_1}-t_{c'_1}|\geq r^{*}$. Note that
	\begin{align*}
	1&=\chi(\omega\,:\,|A_{\omega,\a\wedge\a' c_1}(\Pi_{\omega,\a\wedge\a'
		c_1}(\sigma(\c\b)))| < r^{*}/4)
	\cdot \chi(\omega\,:\,|A_{\omega,\a\wedge\a' c'_1}(\Pi_{\omega,\a\wedge\a'
		c'_1}(\sigma(\c'\b)))| < r^{*}/4)\\
	&+\chi(\omega\,:\,|A_{\omega,\a\wedge\a' c_1}(\Pi_{\omega,\a\wedge\a'
		c_1}(\sigma(\c\b)))| \geq r^{*}/4)
	\cdot \chi(\omega\,:\,|A_{\omega,\a\wedge\a' c'_1}(\Pi_{\omega,\a\wedge\a'
		c'_1}(\sigma(\c'\b)))| < r^{*}/4)\\
	&+\chi(\omega\,:\,|A_{\omega,\a\wedge\a' c_1}(\Pi_{\omega,\a\wedge\a'
		c_1}(\sigma(\c\b)))| <r^{*}/4)
	\cdot \chi(\omega\,:\,|A_{\omega,\a\wedge\a' c'_1}(\Pi_{\omega,\a\wedge\a'
		c'_1}(\sigma(\c'\b)))| \geq r^{*}/4)\\
	&+\chi(\omega\,:\,|A_{\omega,\a\wedge\a' c_1}(\Pi_{\omega,\a\wedge\a'
		c_1}(\sigma(\c\b)))| \geq r^{*}/4)
	\cdot \chi(\omega\,:\,|A_{\omega,\a\wedge\a' c'_1}(\Pi_{\omega,\a\wedge\a'
		c'_1}(\sigma(\c'\b)))| \geq r^{*}/4)\\
	&=:{\chi}_{<}\cdot{\chi}'_{<}+{\chi}_{\geq}\cdot{\chi}'_{<}+{\chi}_{<}\cdot{\chi}'_{\geq}+{\chi}_{\geq}\cdot{\chi}'_{\geq}.
	\end{align*}
	We use this identity to split write \eqref{eq:tosplit} as four summands 
	\[
	\chi_1 =
	\chi_1\cdot{\chi}_{<}\cdot{\chi}'_{<}+
	\chi_1\cdot{\chi}_{\geq}\cdot{\chi}'_{<}+
	\chi_1\cdot{\chi}_{<}\cdot{\chi}'_{\geq}+
	\chi_1\cdot{\chi}_{\geq}\cdot{\chi}'_{\geq}.\]
	The first of these summands is
	\begin{multline*}
	\chi_1 \cdot{\chi}_{<}\cdot{\chi}'_{<}=\chi(\omega\,:\,\Pi_{\omega,\a\wedge\a'}(\c\b)-\Pi_{\omega,\a\wedge\a'}(\c'\b)\in
	B(0,r\cdot C^{1/d}e^{|\a\wedge\a'|(s+\eps)/d}))\\
	\cdot  \chi(\omega\,:\,|A_{\omega,\a\wedge\a' c_1}(\Pi_{\omega,\a\wedge\a'
		c_1}(\sigma(\c\b)))| < r^{*}/4)\\
	\cdot \chi(\omega\,:\,|A_{\omega,\a\wedge\a' c_1'}(\Pi_{\omega,\a\wedge\a'
		c'_1}(\sigma(\c'\b)))| < r^{*}/4)
	\end{multline*}
	Since we are interested in asymptotic behaviour with respect to $Cr^de^{|\a\wedge\a'|(s+\eps)}\to0$, we can without loss of
	generality, assume that $r\cdot C^{1/d}e^{|\a\wedge\a'|(s+\eps)/d} < r^{*}/8$.
	In which case we have $\chi_1 \cdot{\chi}_{<}\cdot{\chi}'_{<}=0.$ This is because, if $\chi_1 \cdot{\chi}_{<}\cdot{\chi}'_{<}=1$ then we would have  
	\begin{align*}r^{*}/2<&|t_{c_1}-t_{c_1'}|-|A_{\omega,\a\wedge\a' c_1}(\Pi_{\omega,\a\wedge\a'
		c_1}(\sigma(\c\b)))-A_{\omega,\a\wedge\a' c'_1}(\Pi_{\omega,\a\wedge\a'
		c'_1}(\sigma(\c'\b)))|\\
	\leq&|A_{\omega,\a\wedge\a' c_1}(\Pi_{\omega,\a\wedge\a'
		c_1}(\sigma(\c\b)))+t_{c_1}-
	A_{\omega,\a\wedge\a' c'_1}(\Pi_{\omega,\a\wedge\a'
		c'_1}(\sigma(\c'\b)))-t_{c_1'}|\\
	=&|\Pi_{\omega,\a\wedge\a'}(\c\b) - \Pi_{\omega,\a\wedge\a'}(\c'\b)|<r^{*}/8,
	\end{align*}
	which is not possible.

Summarising the above, we have shown that the left hand side of inequality \eqref{eq:ineqproof} satisfies
	\begin{align*}
	&\int_{\Omega}\chi_{[-r,r]}(|\Pi_{\omega}(\a\b)-\Pi_{\omega}(\a'\b)|)\cdot
	\chi\left(\omega:\Det(\wA_{\omega,a_1,\ldots,a_n})\in \left(\frac{e^{-n(s+\eps)}}{C},Ce^{-n(s-\eps)}\right),\,\forall1\leq n\leq |\a| \right)\\
	&\hspace{10em}\cdot \chi\left(\omega:\Det(\wA_{\omega,a_1',\ldots,a_n'})\in
	\left(\frac{e^{-n(s+\eps)}}{C},Ce^{-n(s-\eps)} \right),\, \forall1\leq n\leq |\a'|
	\right)\,d\eta\\
	&\leq \int_{\Omega}(\chi_1\cdot{\chi}_{\geq}\cdot{\chi}'_{<}+
	\chi_1\cdot{\chi}_{<}\cdot{\chi}'_{\geq}+
	\chi_1\cdot{\chi}_{\geq}\cdot{\chi}'_{\geq})d\eta + \O(r^d\cdot C\cdot e^{|\a\wedge\a'|(s+\eps)}).
	\end{align*}It remains to appropriately bound the above integral. Manipulating this integral we have 
	\begin{align*}
	&\int_{\Omega}(\chi_1\cdot{\chi}_{\geq}\cdot{\chi}'_{<}+
	\chi_1\cdot{\chi}_{<}\cdot{\chi}'_{\geq}+
	\chi_1\cdot{\chi}_{\geq}\cdot{\chi}'_{\geq})d\eta.\\
	&=\int_{\Omega}(2\chi_1\cdot{\chi}_{\geq}\cdot{\chi}'_{<}+
	\chi_1\cdot{\chi}_{\geq}\cdot{\chi}'_{\geq})d\eta\hspace{15em}\text{(By symmetry)}
	\\
	&\leq 2\int_{\Omega}(\chi_1\cdot{\chi}_{\geq})d\eta\\
	&= 2\int_{\prod_{\d\in\A^*}\Omega_{l(\d)}}(\chi_1\cdot{\chi}_{\geq})\,d\prod_{\d\in\A^*}\eta_{\d}.
	\\
	&=2\int_{\prod_{\d\in\A^*}\Omega_{l(\d)}}
	\chi(\omega\,:\,\Pi_{\omega,\a\wedge\a'}(\c\b)-\Pi_{\omega,\a\wedge\a'}(\c'\b)\in
	B(0,r\cdot C^{1/d}e^{|\a\wedge\a'|(s+\eps)/d}))
	\\
	&\hspace{15em}
	\cdot\chi(\omega\,:\,|A_{\omega,\a\wedge\a' c_1}(\Pi_{\omega,\a\wedge\a'
		c_1}(\sigma(\c\b)))| \geq r^{*}/4)
	\,d\prod_{\d\in\A^*}\eta_{\d}.
	\\
	&=2\int_{\prod_{\d\in\A^*}\Omega_{l(\d)}}
	\chi(\omega\,:\,A_{\omega,\a\wedge\a' c_1}(\Pi_{\omega,\a\wedge\a'
		c_1}(\sigma(\c'\b)))\in
	B(\Pi_{\omega,\a\wedge\a'}(\c'\b)-t_{c_1},r\cdot C^{1/d}e^{|\a\wedge\a'|(s+\eps)/d}))\\&\hspace{15em}
	\cdot\chi(\omega\,:\,|A_{\omega,\a\wedge\a' c_1}(\Pi_{\omega,\a\wedge\a'
		c_1}(\sigma(\c\b)))| \geq r^{*}/4)
	\,d\prod_{\d\in\A^*}\eta_{\d}.
	\end{align*}
	As stated above, there is no loss of generality in assuming that $r\cdot C^{1/d}e^{|\a\wedge\a'|(s+\eps)/d}<r^{*}/8.$ Therefore, if $|A_{\omega,\a\wedge\a' c_1}(\Pi_{\omega,\a\wedge\a'
		c_1}(\sigma(\c\b)))| \geq r^{*}/4$ and $A_{\omega,\a\wedge\a' c_1}(\Pi_{\omega,\a\wedge\a'
		c_1}(\sigma(\c'\b)))\in
	B(\Pi_{\omega,\a\wedge\a'}(\c'\b)-t_{c_1},r\cdot C^{1/d}e^{|\a\wedge\a'|(s+\eps)/d}))$ then we must have $|\Pi_{\omega,\a\wedge\a'}(\c'\b)-t_{c_1}|\geq r^{*}/8.$ Using this fact and then Fubini's theorem, we may bound the integral above by
	\begin{align*}
	&2\int_{\prod_{\d\in\A^*\setminus\{\a\wedge\a'c_1\}}\Omega_{l(\d)}}\int_{\Omega_{c_1}}\chi(\omega\,:\,A_{\omega,\a\wedge\a' c_1}(\Pi_{\omega,\a\wedge\a'
		c_1}(\sigma(\c'\b)))\in
	B(\Pi_{\omega,\a\wedge\a'}(\c'\b)-t_{c_1},r\cdot C^{1/d}e^{|\a\wedge\a'|(s+\eps)/d}))\\&\hspace{15em}
	\cdot\chi(\omega\,:\,|\Pi_{\omega,\a\wedge\a'}(\c'\b)-t_{c_1}|\geq r^{*}/8)\hspace{1em}\,d\eta_{\a\wedge\a'c_1}\,d\prod_{\d\in\A^*}\eta_{\d}.
	\end{align*} 
	Notice that the inner integrand is the probability that the image of a certain point under the linear part of the map associated with
	$\a\wedge\a' c_1$ lies in a certain ball away from the origin.  By our distantly non-singular assumption, this probability is bounded above by a constant times the volume of the
	ball. Therefore we obtain the upper bound 
	\begin{align*}
	&2\int_{\prod_{\d\in\A^*\setminus\{\a\wedge\a'c_1\}}\Omega_{l(\d)}}\int_{\Omega_{c_1}}\chi(\omega\,:\,A_{\omega,\a\wedge\a' c_1}(\Pi_{\omega,\a\wedge\a'
		c_1}(\sigma(\c'\b)))\in
	B(\Pi_{\omega,\a\wedge\a'}(\c'\b)-t_{c_{1}},r\cdot C^{1/d}e^{|\a\wedge\a'|(s+\eps)/d}))\\&\hspace{15em}
	\cdot\chi(\omega\,:\,|\Pi_{\omega,\a\wedge\a'}(\c'\b)-t_{c_{1}}|\geq r^{*}/8)\hspace{1em}\,d\eta_{\a\wedge\a'c_1}\,d\prod_{\stackrel{\d\in\A^*}{\d\neq \a\wedge\a'c_1}}\eta_{\d}\\
	&\leq
	2\int_{\prod_{\d\in\A^*\setminus\{\a\wedge\a'c_1\}}\Omega_{l(\d)}}
	C'\cdot r^d\cdot Ce^{|\a\wedge\a'|(s+\eps)}
	\,d\prod_{\d\in\A^*}\eta_{\d}
	\\
	&\leq 2C'C r^d e^{|\a\wedge\a'|(s+\eps)},
	\end{align*}
	where $C'>0$ is the constant given by the distantly non-singular condition.
	This shows the correct upper bound when $|\a\wedge\a'|\geq 1$.
	The proof where $|\a\wedge\a'|=0$ follows along similar lines, noting that this implies the first
	letters of $\a$ and $\a'$ differ and we can directly apply the distantly non-singular condition. \\

	\noindent
	\textbf{Proof under assumption B.}
	
	The proof under assumption B is similar to the proof under assumption A. However, since we are no longer
	dealing with similarities, the set $(\wA_{\omega,\a\wedge\a'})^{-1}(B(0,r))$ in \eqref{eq:lastline} is not a ball
	but rather an ellipse. 
	Writing $E_{\omega,\a\wedge\a'}=(\wA_{\omega,\a\wedge\a'})^{-1}(B(0,r))$ we obtain 
	\begin{align*}
	&\int_{\Omega}\chi(\omega\,:\,\Pi_{\omega,\a\wedge\a'}(\c\b)-\Pi_{\omega,\a\wedge\a'}(\c'\b)\in
	E_{\omega,\a\wedge\a'})\\
	&\hspace{10em}\cdot\chi\left( \omega\;:\;\Det(\wA_{\omega,\a\wedge\a'})\in\left( \frac{e^{-|\a\wedge\a'|(s+\eps)}}{C},
	Ce^{-|\a\wedge\a'|(s-\eps)} \right)\right)d\eta\\
	&=\int_{\Omega}\chi(\omega\,:\,A_{\omega,\a\wedge\a' c_1}(\Pi_{\omega,\a\wedge\a'
		c_1}(\sigma(\c\b)))\in\Pi_{\omega,\a\wedge\a'}(\c'\b)-t_{c_1}+
	E_{\omega,\a\wedge\a'})\\
	&\hspace{10em}\cdot\chi\left( \omega\;:\;\Det(\wA_{\omega,\a\wedge\a'})\in\left( \frac{e^{-|\a\wedge\a'|(s+\eps)}}{C},
	Ce^{-|\a\wedge\a'|(s-\eps)} \right)\right)d\eta
	\end{align*}
	as an upper bound for the left hand side of \eqref{eq:ineqproof}. The proof then follows by
	an analogous argument where we appeal to Fubini's theorem and the conditions imposed by the
	second characteristic function. It is a consequence of our stronger assumption that the RIFS
	in non-singular that we do not need to include the initial conditioning argument that was
	necessary under assumption $A$.
\end{proof}

\subsection{General results}
To prove Theorem \ref{Main theorem} we will use the following results from \cite{BakOver} and \cite{BerVel2}.

Given $r>0$, we say that a set $Y\subset \mathbb{R}^d$ is an $r$-separated set if $|z-z'|>r$ for all
distinct $z,z'\in Y$. Given a finite $Y\subset\mathbb{R}^d$ and $r>0$ we let $$T(Y,r):=\sup\{\# Y':Y'\subset Y \textrm{ and }Y' \textrm{ is an }r\textrm{-separated set}\}.$$ Now suppose that we have a metric space $\Omega$ and $\widetilde{X}$ is some compact subset of $\mathbb{R}^d$. Suppose that for each $n\in \mathbb{N}$ there exists a finite set of functions $\{f_{l,n}:\Omega\to \widetilde{X}\}_{l=1}^{R_n}.$ For each $\omega\in \Omega$ we let $$Y_{n}(\omega):=\{f_{l,n}(\omega)\}_{l=1}^{R_n}.$$ Moreover, given $c>0,s>0,$ and $n\in\mathbb{N},$ we let $$B(c,s,n):=\left\{\omega\in \Omega: \frac{T(Y_{n}(\omega),\frac{s}{R_{n}^{1/d}})}{R_n}>c\right\}.$$ The following proposition was proved in \cite{BakOver}.

\begin{proposition}
	\label{general prop}
	Let $\omega\in \Omega$ and $g:\mathbb{N}\to [0,\infty)$. Assume that the following properties are satisfied:
	\begin{itemize}
		\item There exists $\gamma>1$ such that 
		\[
		R_n\asymp \gamma^n.
		\]
		\item There exists $c>0$ and $s>0$ such that 
		\[
		\sum_{\substack{n\in\N\\\omega\in B(c,s,n)}}g(n)=\infty.
		\]
	\end{itemize}
	Then $$\left\{x\in \mathbb{R}^d:|x-f_{l,n}(\omega)|\leq \left(\frac{g(n)}{R_n}\right)^{1/d}
	\textrm{ for i.m.\ }(l,n)\in \{1,\ldots, R_n\}\times\mathbb{N}\right\}$$
	has positive Lebesgue measure.
\end{proposition}

We will also use the following lemma which follows from Lemma 1 of \cite{BerVel2}.
\begin{lemma}
	\label{comparable balls}
	Let $(x_j)$ be a sequence of points in $\mathbb{R}^d$ and $(r_j),(r_j')$ be two sequences of
	positive real numbers both converging to zero. If $r_j\asymp r_j'$ then $$\L(x:x\in
	B(x_j,r_j) \textrm{ for i.m.\ } j)=\L(x:x\in B(x_j,r_j') \textrm{ for i.m.\ } j).$$
\end{lemma}

\section{Proof of Theorem \ref{Main theorem}}
\label{proof section}
In this section we will prove Theorem \ref{Main theorem}. We begin by remarking that it is a
consequence of Lemma \ref{comparable balls} that Statement 2.\ follows from Statement 1. To prove
Theorem \ref{Main theorem} it therefore suffices to prove Statement 1.

For the rest of this section we fix a RIFS satisfying either assumption A or assumption B, we fix $\m$ a slowly decaying $\sigma$-invariant ergodic probability measure such that $\frac{h(\m)}{\lambda(\eta,\m)}>1$, and $\b\in \A^{\N}$ is fixed. We also let $\eps_1>0$ be sufficiently small such that 
\begin{equation}
\label{eps1def}
h(\m)-\lambda(\m,\eta)-2\eps_1>0.
\end{equation} 
Such an $\eps_1>0$ must exist because of our assumption that $\frac{h(\m)}{\lambda(\m,\eta)}>1.$

Let us fix a parameter $0<\eps_0 <1/2$. By Proposition \ref{important properties prop} we may fix a large $N_{1}'=N_{1}'(\eta)\in \N$ and $N_{2}'=N_{2}'(\eta)\in \N$ such that for a set of $\omega$ with $\eta$-measure at
least $1-\eps_0,$ there exists $\widetilde{L}_{\m,n,\eps_{1}}(\omega)\subseteq L_{\m,n,\eps_{1}}$  for each $n\geq
N_{2}'$ that satisfies the following:
\begin{enumerate}	
	\item For each $\a\in \widetilde L_{\m,n,\eps_{1}}(\omega)$ we have $$|\Det(\wA_{\omega,a_1,\ldots,a_n})|\in \left(\frac{e^{-n(\lambda(\eta,\m)+\eps_1)}}{C_1},C_1e^{-n(\lambda(\eta,\m)-\eps_1)} \right)$$ for all $ N_1'\leq n \leq |\a|$.
	\item $\# \widetilde{L}_{\m,n,\eps_{1}}(\omega)\asymp c_{\m}^{-n}.$ 
\end{enumerate}We denote the set of $\omega$ for which these properties hold by $\Omega'$. By construction $\eta(\Omega')\geq 1-\eps_0$. We will show that the conclusion of Statement $1.$ from Theorem \ref{Main theorem} is satisfied by almost every element of $\Omega'$. Since $\eps_0$ is arbitrary this will complete our proof. 

Note that since $N_{1}'$ only depends upon $\eta$ we can in fact strengthen Property $1.$ on the set $\Omega'$. By letting $C_{1}$ depend upon $\Omega'$, we can replace Property $1$ with the following stronger statement:
\begin{enumerate}
	\item[3.] For each $\a\in \widetilde L_{\m,n,\eps_{1}}(\omega)$ we have $$|\Det(\wA_{\omega,a_1,\ldots,a_n})|\in \left(\frac{e^{-n(\lambda(\eta,\m)+\eps_1)}}{C_1},C_1e^{-n(\lambda(\eta,\m)-\eps_1)} \right)$$ for all $ 1\leq n \leq |\a|$.
\end{enumerate}

The following proposition tells us that for a typical $\omega\in \Omega'$ there are not too many $(\a,\a')\in\widetilde L_{\m,n,\eps_{1}}(\omega)\times \widetilde L_{\m,n,\eps_{1}}(\omega)$ for which $\Pi_{\omega}(\a\b)$ and $\Pi_{\omega}(\a'\b)$ are close. The proof is based upon arguments given in \cite{BakOver}, which in turn are an appropriate adaptation of arguments due to \cite{BenSol} and \cite{PerSol}.

\begin{proposition}
	\label{Pairs prop}
	Let $\Omega'$ be as above. For any $s>0$ and $n\geq N_{2}'$ we have 
	\[
	\int_{\Omega'}\frac{\#\{(\a,\a')\in \widetilde{L}_{\m,n,\eps_{1}}(\omega)\times \widetilde{L}_{\m,n,\eps_{1}}(\omega) :
		|\Pi_{\omega}(\a\b)-\Pi_{\omega}(\a'\b)|\leq \frac{s}{\#L_{\m,n}^{1/d}},\, \a\neq
		\a'\}}{\# L_{\m,n}}d\eta=\O(s^d).
	\]
\end{proposition}
\begin{proof}
	We begin by observing that 	
	\begin{align*}
	&\#\left\{(\a,\a')\in \widetilde{L}_{\m,n,\eps_{1}}(\omega)\times \widetilde{L}_{\m,n,\eps_{1}}(\omega) : |\Pi_{\omega}(\a\b)-\Pi_{\omega}(\a'\b)|\leq \frac{s}{\#L_{\m,n}^{1/d}},\, \a\neq \a'\right\}\\
	=&\sum_{\stackrel{(\a,\a')\in \widetilde{L}_{\m,n,\eps_{1}}(\omega)\times \widetilde{L}_{\m,n,\eps_{1}}(\omega) }{\a\neq \a'}}\chi_{[\frac{-s}{\#L_{\m,n}^{1/d}},\frac{s}{\#L_{\m,n}^{1/d}}]}(|\Pi_{\omega}(\a\b)-\Pi_{\omega}(\a'\b)|).
	\end{align*}By Property $3.$ above we know that for any $\omega\in\Omega'$ each $\a\in \widetilde{L}_{\m,n,\eps_{1}}(\omega)$ satisfies 
	\begin{equation}
	\label{good determinants}
	|\Det(\wA_{\omega,a_1,\ldots,a_n})|\in \left(\frac{e^{-n(\lambda(\eta,\m)+\eps_1)}}{C_1},C_1e^{-n(\lambda(\eta,\m)-\eps_1)} \right)
	\end{equation} for all $ 1\leq n \leq |\a|$. Therefore we have the following upper bound for our counting function
	\begin{align*}
	&\#\left\{(\a,\a')\in \widetilde{L}_{\m,n,\eps_{1}}(\omega)\times \widetilde{L}_{\m,n,\eps_{1}}(\omega) : |\Pi_{\omega}(\a\b)-\Pi_{\omega}(\a'\b)|\leq \frac{s}{\#L_{\m,n}^{1/d}},\, \a\neq \a'\right\}\\
	\leq & \sum_{\stackrel{(\a,\a')\in L_{\m,n,\eps_{1}}\times L_{\m,n,\eps_{1}} }{\a\neq \a'}}\chi_{[\frac{-s}{\#L_{\m,n}^{1/d}},\frac{s}{\#L_{\m,n}^{1/d}}]}(|\Pi_{\omega}(\a\b)-\Pi_{\omega}(\a'\b)|)\cdot \chi(\omega:\a \textrm{ satisfies }\eqref{good determinants})\\
	& \cdot \chi(\omega:\a' \textrm{ satisfies }\eqref{good determinants})
	\end{align*}
	Recall that $L_{\m,n,\eps_{1}}$ was defined in \eqref{uniform entropy condition}. Notice that we are now summing over all pairs in $L_{\m,n,\eps_{1}}\times L_{\m,n,\eps_{1}}$ such that
	$\a\neq \a'$. In particular the terms in this sum no longer depend upon $\omega$. Since $\m([\a])\asymp \# L_{\m,n}^{-1}$ for each $\a\in L_{\m,n}$ we have 
	\begin{align*}
	&\int_{\Omega'}\frac{\#\{(\a,\a')\in \widetilde{L}_{\m,n,\eps_{1}}(\omega)\times \widetilde{L}_{\m,n,\eps_{1}}(\omega) : |\Pi_{\omega}(\a\b)-\Pi_{\omega}(\a'\b)|\leq \frac{s}{\#L_{\m,n}^{1/d}},\, \a\neq \a'\}}{\# L_{\m,n}}d\eta
	\\
	\ll & \#L_{\m,n}\sum_{\stackrel{(\a,\a')\in L_{\m,n,\eps_{1}}\times L_{\m,n,\eps_{1}} }{\a\neq \a'}}\m([\a])\m([\a'])\int_{\Omega'} \chi_{[\frac{-s}{\#L_{\m,n}^{1/d}},\frac{s}{\#L_{\m,n}^{1/d}}]}(|\Pi_{\omega}(\a\b)-\Pi_{\omega}(\a'\b)|)
	\\
	&\hspace{17em}\cdot \chi(\omega:\a \textrm{ satisfies }\eqref{good determinants})\cdot
	\chi(\omega:\a' \textrm{ satisfies }\eqref{good determinants})\;d\eta
	\\
	\ll  & \#L_{\m,n}\sum_{\a\in L_{\m,n,\eps_{1}}}\sum_{k=0}^{|\a|}\sum_{\stackrel{\a'\in
			L_{\m,n,\eps_{1}}}{|\a\wedge\a'|=k}}\m([\a])\m([\a'])\int_{\Omega}
	\chi_{[\frac{-s}{\#L_{\m,n}^{1/d}},\frac{s}{\#L_{\m,n}^{1/d}}]}(|\Pi_{\omega}(\a\b)-\Pi_{\omega}(\a'\b)|)
	\\&\hspace{17em}\cdot \chi(\omega:\a \textrm{ satisfies }\eqref{good determinants})
	\cdot \chi(\omega:\a' \textrm{ satisfies }\eqref{good determinants})d\eta.
	\end{align*} 
	The integrals appearing in the sum above are in a form where we can apply Lemma \ref{transversality
		lemma rework}. Applying Lemma~\ref{transversality lemma rework} and the definition of $L_{\m,n,\eps_{1}}$, we see that we can bound the above by
	\begin{align*}
	&C\#L_{\m,n}\sum_{\a\in L_{\m,n,\eps_{1}}}\sum_{k=0}^{|\a|}\sum_{\stackrel{\a'\in
			L_{\m,n,\eps_{1}}}{|\a\wedge\a'|=k}}\m([\a])\m([\a'])\frac{s^dC_1 e^{k(\lambda(\eta,\m)+\eps_1)}}{\#L_{\m,n}}\\
	&\ll s^d\sum_{\a\in L_{\m,n,\eps_{1}}}\m([\a])\sum_{k=0}^{|\a|}\sum_{\stackrel{\a'\in
			L_{\m,n,\eps_{1}}}{|\a\wedge\a'|=k}}\m([\a'])e^{k(\lambda(\eta,\m)+\eps_1)}\\
	&\ll s^d\sum_{\a\in L_{\m,n,\eps_{1}}}\m([\a])\sum_{k=0}^{|\a|}\m([a_1\ldots a_{k}])e^{k(\lambda(\eta,\m)+\eps_1)}\\
	&\ll s^d\sum_{\a\in L_{\m,n,\eps_{1}}}\m([\a])\sum_{k=0}^{|\a|}e^{-k(h(\m)-\eps_1)}e^{k(\lambda(\eta,\m)+\eps_1)}\\
	&\ll s^d\sum_{\a\in L_{\m,n,\eps_{1}}}\m([\a])\sum_{k=0}^{|\a|}e^{-k(h(\m)-\lambda(\eta,\m)-2\eps_1)}\\
	&\ll s^d\sum_{\a\in L_{\m,n,\eps_{1}}}\m([\a])\\
	&\ll s^d.
	\end{align*}
	In the penultimate line we used that  $\sum_{k=0}^{\infty}e^{-k(h(\m)-\lambda(\eta,\m)-2\eps_1)}<\infty$. This is a consequence of the definition of $\eps_1$. Since all constants are universal, the proof follows.
\end{proof}
We now show how Proposition \ref{Pairs prop} can be used to construct a large separated subset of projections for a large set of $n$ for almost every $\omega\in \Omega'$.

For each $n\in \N$ and $\omega\in \Omega'$ we let $$Y_{n}(\omega):=\{\Pi_{\omega}(\a\b)\}_{\a\in \widetilde{L}_{\m,n,\eps_{1}}(\omega)}.$$
Moreover, given $s>0$, $\omega\in \Omega'$, and $n\geq N_{2}'$ we let $$\CP(s,\omega,n):=\left\{(\a,\a')\in \widetilde{L}_{\m,n,\eps_{1}}(\omega)\times \widetilde{L}_{\m,n,\eps_{1}}(\omega) : |\Pi_{\omega}(\a\b)-\Pi_{\omega}(\a'\b)|\leq \frac{s}{\#L_{\m,n}^{1/d}},\, \a\neq \a'\right\}$$ Recall that $T(Y,r)$ is the maximal cardinality of $r$-separated subsets of $Y$.
We will need the following technical result.

\begin{lemma}For any $\omega\in \Omega'$ and $n\in N_{2}'$ we have
	\label{well separated pairs}
	$$\#\widetilde{L}_{\m,n,\eps_{1}}(\omega)\leq T\left(Y_{n}(\omega),\frac{s}{\# L_{\m,n}^{1/d}}\right) +\#\CP(s,\omega,n).$$
\end{lemma}
\begin{proof}
	We start by observing that 
	\begin{align*}
	\widetilde{L}_{\m,n,\eps_{1}}(\omega)=&\left\{\a\in \widetilde{L}_{\m,n,\eps_{1}}(\omega):|\Pi_{\omega}(\a\b)-\Pi_{\omega}(\a'\b)|> \frac{s}{\#L_{\m,n}^{1/d}}\, \forall \a'\neq \a\right\}\\
	\cup & \left\{\a\in \widetilde{L}_{\m,n,\eps_{1}}(\omega):\exists \a'\neq \a \textrm{ s.t.\ } |\Pi_{\omega}(\a\b)-\Pi_{\omega}(\a'\b)|\leq  \frac{s}{\#L_{\m,n}^{1/d}}\right\}
	\end{align*} 
	is a disjoint union. Notice also that the set of images corresponding to those $\a$ belonging to the first set in this union is $\frac{s}{\# L_{\m,n}^{1/d}}$-separated. Therefore
	$$\# \left\{\a\in \widetilde{L}_{\m,n,\eps_{1}}(\omega):|\Pi_{\omega}(\a\b)-\Pi_{\omega}(\a'\b)|>
	\frac{s}{\#L_{\m,n}^{1/d}}\, \forall \a'\neq \a\right\}\leq
	T\left(Y_{n}(\omega),\frac{s}{\# L_{\m,n}^{1/d}}\right).$$ Similarly, for the second set in this
	union we have
	$$\left\{\a\in \widetilde{L}_{\m,n,\eps_{1}}(\omega):\exists \a'\neq \a \textrm{ s.t.\ }
	|\Pi_{\omega}(\a\b)-\Pi_{\omega}(\a'\b)|\leq  \frac{s}{\#L_{\m,n}^{1/d}}\right\}\leq
	\# \CP(s,\omega,n).$$ This follows because the map $(\a,\a')\to \a$ from $\CP(s,\omega,n)$ to this
	set is surjective.  The desired inequality now follows. 
\end{proof}

Given $n\in\N$ and $\omega\in \Omega'$ let $$Y_{n}'(\omega):=\{\Pi_{\omega}(\a\b)\}_{\a\in L_{\m,n}}.$$ Notice that $Y_{n}(\omega)\subset Y_{n}'(\omega)$ therefore $T(Y_{n}(\omega),r)\leq  T(Y_{n}'(\omega),r)$ for any $r>0$. Given $c>0, s>0,$ and $n\in \N$ we also let 
$$B(c,s,n):=\left\{\omega\in \Omega':\frac{T(Y_{n}'(\omega),\frac{s}{\#L_{\m,n}^{1/d}})}{\# L_{\m,n}}>c\right\}.$$ Recall that we define the upper density of a set $B\subset \N$ to be
\[
\overline{d}(B):=\limsup_{n\to\infty}\frac{\#\{1\leq j\leq n:j\in B\}}{n}.
\]
\begin{proposition}The following equality holds
	\label{full measure prop}
	$$\eta\left(\bigcap_{\eps>0}\bigcup_{c,s>0}\{\omega\in \Omega':\overline{d}(n:\omega\in B(c,s,n))\geq 1-\eps\}\right)=\eta(\Omega').$$
\end{proposition}

\begin{proof}
	Let $\eps>0$ be arbitrary. Notice that by Proposition \ref{Pairs prop} and Markov's inequality,  for any $c>0$ we have 
	\[
	c\cdot \eta(\omega\in \Omega':\#\CP(s,\omega,n)\geq c\#L_{\m,n})=\O(s^d)
	\]
	for any $n\geq N_2'.$
	Therefore, since $\#L_{\m,n}\asymp \#\widetilde{L}_{\m,n,\eps_{1}}(\omega)$ we can choose $c,s>0$ such that
	\[
	\eta(\omega\in \Omega':\#\CP(s,\omega,n)\geq c\#\widetilde{L}_{\m,n,\eps_{1}}(\omega))<\eps.
	\]
	Therefore by Lemma \ref{well separated pairs}, for this choice of $c,s$ we have
	$$\eta\left(\omega\in \Omega':T\left(Y_{n}(\omega),\frac{s}{\# L_{\m,n}^{1/d}}\right)\geq
	\#\widetilde{L}_{\m,n,\eps_{1}}(\omega)(1-c)\right)\geq \eta(\Omega')-\eps$$ for any $n\geq N_2'$. Using this
	inequality and apply Fatou's lemma we have
	\begin{align}
	\label{nearly full}
	&\int_{\Omega'}\overline{d}\left(n:T\left(Y_{n}(\omega),\frac{s}{\# L_{\m,n}^{1/d}}\right)\geq
	\#\widetilde{L}_{\m,n,\eps_{1}}(\omega)(1-c)\right)d\,\eta\nonumber \\
	=&\int_{\Omega'}\limsup_{N\to\infty}\frac{\#\{1\leq n\leq N: T\left(Y_{n}(\omega),\frac{s}{\# L_{\m,n}^{1/d}}\right)
		\geq \#\widetilde{L}_{\m,n,\eps_{1}}(\omega)(1-c)\}}{N}\, d\eta\nonumber\\
	\geq & \limsup_{N\to\infty}\int_{\Omega'}\frac{\sum_{n=1}^{N}\chi(\omega:T\left(Y_{n}(\omega),\frac{s}{\# L_{\m,n}^{1/d}}\right)\geq \#\widetilde{L}_{\m,n,\eps_{1}}(\omega)(1-c))}{N}\, d\eta\nonumber\\
	\geq & \eta(\Omega')-\eps.
	\end{align}
	We now show that this implies that the occurrence of a large separated set for a set of $n$ with high upper density has large probability. That is, we will prove the inequality
	\begin{equation}\label{eq:toprovecontra}
	\eta\left(\omega\in \Omega':\overline{d}\left(n:T\left(Y_{n}(\omega),\frac{s}{\# L_{\m,n}^{1/d}}\right)\geq \#\widetilde{L}_{\m,n,\eps_{1}}(\omega)(1-c)\right)\geq 1-\sqrt{\eps}\right)\geq \eta(\Omega')-\sqrt{\eps}.
	\end{equation}
	Assume for a contradiction that 
	\[
	\eta\left(\omega\in \Omega':\overline{d}\left(n:T\left(Y_{n}(\omega),\frac{s}{\# L_{\m,n}^{1/d}}\right)\geq \#\widetilde{L}_{\m,n,\eps_{1}}(\omega)(1-c)\right)< 1-\sqrt{\eps}\right)>\sqrt{\eps}.
	\]
	As the density is always bounded above by $1$ we have
	\begin{align}
	&\int_{\Omega'}\overline{d}\left(n:T\left(Y_{n}(\omega),\frac{s}{\# L_{\m,n}^{1/d}}\right)\geq
	\#\widetilde{L}_{\m,n,\eps_{1}}(\omega)(1-c)\right)d\,\eta\nonumber \\
	&\leq \eta\left(\omega\in \Omega':\overline{d}\left(n:T\left(Y_{n}(\omega),\frac{s}{\#
		L_{\m,n}^{1/d}}\right)\geq \#\widetilde{L}_{\m,n,\eps_{1}}(\omega)(1-c)\right)<
	1-\sqrt{\eps}\right)(1-\sqrt{\eps})\nonumber\\
	&+ \eta\left(\omega\in\Omega':\overline{d}\left(n:T\left(Y_{n}(\omega),\frac{s}{\#
		L_{\m,n}^{1/d}}\right)\geq \#\widetilde{L}_{\m,n,\eps_{1}}(\omega)(1-c)\right)\geq
	1-\sqrt{\eps}\right).\label{eq:presimplify}
	\end{align}
	The second term on the right hand side of \eqref{eq:presimplify} is equal to
	\[
	\eta(\Omega')- \eta\left(\omega\in \Omega':\overline{d}\left(n:T\left(Y_{n}(\omega),\frac{s}{\#
		L_{\m,n}^{1/d}}\right)\geq \#\widetilde{L}_{\m,n,\eps_{1}}(\omega)(1-c)\right)< 1-\sqrt{\eps}\right).
	\]
	Therefore the right hand side of \eqref{eq:presimplify} can be bounded above by
	\begin{align*}&\eta(\Omega')-\sqrt{\eps}\eta\left(\omega\in\Omega':\overline{d}\left(n:T\left(Y_{n}(\omega),\frac{s}{\#
		L_{\m,n}^{1/d}}\right)\geq \#\widetilde{L}_{\m,n,\eps_{1}}(\omega)(1-c)\right)< 1-\sqrt{\eps}\right)\\
	&<\eta(\Omega')-\eps.
	\end{align*} Where in the final line we used our underlying assumption. However, this contradicts \eqref{nearly full}. Therefore \eqref{eq:toprovecontra} holds.
	
	We have proved that for any $\eps>0$ we can chose $c,s>0$ such that
	\eqref{eq:toprovecontra} holds. In particular, letting $\eps_k\to 0$ with $\eps_k <\eps$ and picking appropriate
	sequences $s_k,c_k$ we may conclude that
	\[
	\eta\left(\bigcup_{c,s>0}\{\omega\in\Omega':\overline{d}\left(n:T\left(Y_{n}(\omega),\frac{s}{\# L_{\m,n}^{1/d}}\right)\geq
	\#\widetilde{L}_{\m,n,\eps_{1}}(\omega)(1-c)\right)\geq 1-\eps\}\right)=\eta(\Omega').
	\]
	Recall that $\eps>0$ was arbitrary. 
	Therefore taking the
	intersection over all $\eps>0$ we have
	\begin{equation}
	\label{nearly done}
	\eta\left(\bigcap_{\eps>0}\bigcup_{c,s>0}\{\omega\in
	\Omega':\overline{d}\left(n:T\left(Y_{n}(\omega),\frac{s}{\# L_{\m,n}^{1/d}}\right)\geq
	\#\widetilde{L}_{\m,n,\eps_{1}}(\omega)(1-c)\right)\geq 1-\eps\}\right)=\eta(\Omega').
	\end{equation}
	
	Because $T\left(Y_{n}(\omega),\frac{s}{\# L_{\m,n}^{1/d}}\right)\leq
	T\left(Y_{n}'(\omega),\frac{s}{\# L_{\m,n}^{1/d}}\right)$ and  $\# L_{\m,n}\asymp \#
	\widetilde{L}_{\m,n}$, there exists $K>0$ such that
	\[
	T\left(Y_{n}(\omega),\frac{s}{\# L_{\m,n}^{1/d}}\right)\geq
	\#\widetilde{L}_{\m,n}(\omega)(1-c)
	\Rightarrow T\left(Y'_{n}(\omega),\frac{s}{\# L_{\m,n}^{1/d}}\right)\geq \frac{\# L_{\m,n}(1-c)}{K}.
	\]
Therefore by \eqref{nearly done} we have
	\[
	\eta\left(\bigcap_{\eps>0}\bigcup_{c,s>0}\{\omega\in
	\Omega':\overline{d}\left(n:\omega\in B(\tfrac{1-c}{K},s,n)\right)\geq 1-\eps\}\right)=\eta(\Omega').
	\]
	This completes our proof.
\end{proof}
With Proposition \ref{full measure prop} we are now in a position to prove Statement 1 from Theorem \ref{Main theorem}.

\begin{proof}[Theorem \ref{Main theorem}, Statement $1$]
	Let us fix 	$$\omega\in \bigcap_{\eps>0}\bigcup_{c,s>0}\{\omega\in \Omega':\overline{d}(n:\omega\in B(c,s,n))\geq 1-\eps\}.$$ Let $g\in G$ be arbitrary. By definition there exists $\eps^{*}>0$ such that $g\in G_{\eps^{*}}$. By the definition of  $G_{\eps^{*}}$ we can choose $c>0,s>0$ such that 
	\begin{equation*}
	\sum_{n:\omega\in B(c,s,n)}g(n)=\infty.
	\end{equation*} Moreover, notice that $\#L_{\m,n}\asymp c_{\m}^{-n}$. Therefore the two
	assumptions of Proposition \ref{general prop} are satisfied and so the set $$\left\{x\in
	\mathbb{R}^{d}:|x-\Pi_{\omega}(\a\b)|\leq \left(\frac{g(n)}{\#
      L_{\m,n}}\right)^{1/d}\textrm{ for some }\a\in L_{\m,n}\textrm{ for i.m.\ }n\right\}$$ has
      positive Lebesgue measure. Recall that $\#L_{\m,n}^{-1}\asymp \m([\a])$ for any $\a\in
      L_{\m,n}.$ Therefore Lemma \ref{comparable balls} implies that $U_{\omega}(\b,\m,g)$ has
      positive Lebesgue measure for any $g\in G$. By Proposition \ref{full measure prop} it follows
      that for almost every $\omega\in\Omega'$ the set $U_{\omega}(\b,\m,g)$ has positive Lebesgue
      measure for any $g\in G$. Since $\eta(\Omega')>1-\epsilon_0,$ and $\epsilon_0$ was arbitrary,
      it follows that for almost every $\omega\in \Omega,$ for any $g\in G$ the set
      $U_{\omega}(\b,\m,g)$ has positive Lebesgue measure. This completes our proof.
\end{proof}

\section{Proof of Corollary \ref{Bernoulli corollary}}\label{sec:Corollary}
We now show how Corollary \ref{Bernoulli corollary} follows from Theorem \ref{Main theorem}. The proof is essentially the same as the proof of Corollary 2.3 from \cite{BakOver}. We include the details for completion.
\begin{proof}[Proof of Corollary \ref{Bernoulli corollary}]
	Let us fix a RIFS such either assumption A or assumption B is satisfied. Let us also fix a probability vector $(p_i)_{i\in \A}$ such that $\frac{-\sum_{i\in A}p_i\log p_i}{\sum p_i\lambda'(\eta_i)}>1.$

	By Theorem \ref{Main theorem}, to prove our result it suffices to show
	that if we let $\Psi:\A^*\to[0,\infty)$ be given by
	$\Psi(\a)=\frac{\prod_{k=1}^{|\a|}p_{a_k}}{|\a|},$ then $\Psi$ is equivalent to $(\m,g)$ for some
	$g\in G$. Here we let $\m$ denote the Bernoulli measure corresponding to $(p_i)_{i\in \A}$. 
	
	Let $g(n)=\frac{1}{n},$ then using the well known identity 
	\[
	\sum_{n=1}^{N}\frac{1}{n}= \log N
	+\O(1),
	\]
	it can be shown that $g\in G$. For any $\a\in \A^*$ we have 
	\begin{equation}
	\label{decay bounds}
	(\min_{i\in \A}p_i)^{|\a|}\leq \m([\a])\leq (\max_{i\in \A}p_i)^{|\a|}.
	\end{equation}Using \eqref{decay bounds} and the fact each $\a\in L_{\m,n}$ satisfies
	$\m([\a])\asymp c_{\m}^n,$ we may deduce that 
	\[
	|\a|\asymp n
	\]
	for any $\a\in L_{\m,n}$. This implies
	\[
	\frac{\prod_{k=1}^{|\a|}p_{a_k}}{|\a|}\asymp \frac{\m([\a])}{n}
	\]
	for any $\a\in L_{\m,n}$.
	Therefore $\Psi$ is equivalent to $(\m,g)$ for our choice of $g$. This completes our proof. 	
\end{proof}

\section{Examples}
\label{example section}
In this section we detail some examples of RIFSs to which our results can be applied. The first example
is stochastically self-similar and is distantly non-singular, whereas the second is stochastically
self-affine and non-singular.

\begin{example}
	For each $i\in \A$ assume that there exists $0\leq r_{i}^{-}<r_{i}^{+}<1$ such that
	\[
	\Omega_i:=\{\lambda \cdot O: \lambda\in [r_i^{-},r_{i}^{+}]\text{ and } O\in \O(d)\}.
	\]
	Here
	$\O(d)$ is the set of $d\times d$ orthogonal matrices. Note that $\Omega_i\subset S_d$ for all
	$i\in \A$. For each $i\in \A$ we define a measure $\eta_i$ on $\Omega_i$ according to the law
	where $\lambda$ and $O$ are chosen independently with respect to the normalised Lebesgue measure
	on $[r_i^{-},r_{i}^{+}]$ and the Haar measure on $\O(d)$ respectively. For any $x\in \mathbb{R}^d$
	we define the map $P_x:\Omega_i\to \mathbb{R}^d$ given by $P_x(A)=Ax$. It can be shown that the
	pushforward measure $(P_x)_{*}\eta_i$ is the normalised Lebesgue measure on the annulus 
	\[
	\{y\in
	\mathbb{R}^d: r_{i}^{-}\|x\|\leq \|y\|\leq r_{i}^{+}\|x\|\}.
	\]
	Note that $(P_x)_{*}\eta_i$  is
	absolutely continuous with respect to the Lebesgue measure on $\mathbb{R}^d$. Moreover, for any
	$\eps>0,$ there exists $C=C(\eps)>0$ such that for any $x$ satisfying $\|x\|\geq
	\eps,$ the Radon-Nikodym derivative of $(P_x)_{*}\eta_i$ is uniformly bounded above by $C$.
	
	Now let us fix a collection $\{t_i\}_{i\in \A}$ of distinct translation vectors and let
	$r_0:=\frac{\min_{i\neq j}|t_i-t_j|}{16}$. If $\|x\|< r_0$ and $r<r_0$ then for any $y\in
	\R^d\setminus B(0,\frac{\min_{i\neq j}|t_i-t_j|}{8})$ we have $\{A\in\Omega_i\;:\;A(x) \in
	B(y,r)\}=\emptyset$. Therefore 
	\begin{equation}
	\label{one}
	\eta_i(\A\in\Omega_i\;:\;A(x) \in B(y,r))=0.
	\end{equation}If $\|x\|< r_0$ and $r\geq r_0,$ we can choose $C_1>0$ sufficiently large in a way that only depends upon $r_0$ such that  
	\begin{equation}
	\label{two}
	\eta_i(\A\in\Omega_i\;:\;A(x) \in B(y,r))\leq C_1 r^d
	\end{equation} for all $y\in \mathbb{R}^d$. If $\|x\|\geq r_0$ then it follows by our above remarks regarding the Radon-Nikodym derivative that there exists $C_2>0$ independent of $x$ such that
	\begin{equation}
	\label{three}
	\eta_i(\A\in\Omega_i\;:\;A(x) \in B(y,r))=((P_x)_{*}\eta_i)(B(y,r))\leq C_2 r^d
	\end{equation}for all $y\in \mathbb{R}^d$ and $r>0$.
	
	Combining \eqref{one}, \eqref{two}, and \eqref{three} we see that the RIFS  
	$(\{\Omega_i\}_{i\in \A},\{\eta_i\}_{i\in \A},\{t_i\}_{i\in \A})$ is distantly non-singular.
	We need to check that the logarithmic condition \eqref{eq:CramerAss} holds for $|s|$ sufficiently small. We see that
	\[
	  \log\int_{\Omega_i}|\Det(A)|^sd\eta_i(A) = \log \int_{r_i^-}^{r_i^+} \frac{r^s}{r_i^+ - r_i^-} dr
	  =\log\frac{(r_i^+)^{s+1} - (r_i^-)^{s+1}}{(r_i^+-r_i^-)(s+1)},
	\]
	which is finite for all $s>-1$. Therefore \eqref{eq:CramerAss} holds for all $|s|<1$.

	Now by an
	appropriate choice of parameters, it is straightforward to construct many slowly decaying
	$\sigma$-invariant ergodic probability measures $\m$ such that $\frac{h(\m)}{\lambda(\eta,\m)}>1$.
	Therefore Theorem \ref{Main theorem}, Corollary \ref{Bernoulli corollary}, and Corollary \ref{measure corollary} can be applied to this random model.

\end{example}

\begin{example}
	For each $i\in \A$ let $Z_i$ be a compact subset of $M_d$ such
	that each $A\in Z_i$ satisfies $\|Ax\|\geq c_i\|x\|$ for all $x\in \R^d$ for some $c_i>0$. Also assume that for each
	$i\in \A$ there exists a Borel probability measure $\nu_i$ supported on $Z_i$. For each $i\in \A$
	let $0\leq r_{i}^{-}<r_{i}^{+}<1$ and 
	\[
	\Omega_i:= \{A=\lambda\cdot O B: \lambda\in
	[r_{i}^{-},r_{i}^{+}], O\in \O(d), \text{ and }B\in Z_i\}.
	\]
	We define a measure $\eta_i$ on
	$\Omega_i$ by choosing $\lambda$, $O,$ and $B$ independently with respect to the normalised
	Lebesgue measure $\L$ on $[r_{i}^{-},r_{i}^{+}],$ the Haar measure $m$ on $\O(d),$ and $\nu_i$
	respectively.	
	Let $\{t_i\}_{i\in \A}$ be a finite set of distinct vectors. We assume $\{t_i\}_{i\in \A}$ and
	$\{\Omega_i\}$ are such that there exists $\delta>0$ for which 
	\[
	B(0,\delta)\cap \bigcup_{\omega\in
		\Omega}\Pi_{\omega}(\A^{\mathbb{N}})=\emptyset.
	\]
	This property is satisfied for example if each
	$t_i$ satisfies $\|t_i\|=1$ and each $A\in \cup_{i\in \A}\Omega_i$ satisfies $\|A\|<1/2$.

	We now show that if the above conditions are satisfied then the RIFS is non-singular.
	Let us fix an ellipse $E$ and $x\in \cup_{\omega}\Pi_{\omega}(\A^{\N})$. By Fubini's
	theorem
	\begin{equation}
	\label{1a}
	\eta_i(A\in\Omega_i\;:\;A(x) \in E )=\int_{Z_i}(\L\times m )((r,O): r\cdot OBx\in E)d\nu_i(B).
	\end{equation} By construction $Bx$ is a point with norm $\|Bx\|\geq c\delta$. Therefore by the same
	reasoning as given in Example $1,$ by considering appropriate pushforwards, it can be shown that
	there exists $C>0$ independent of $x$ and $i$ such that 
	\[
	(\L\times m )((r,O): r\cdot OBx\in E)\leq C \cdot Vol(E)
	\]
	for any $B\in Z_i$. Substituting this
	bound into \eqref{1a} we obtain 
	\[
	\eta_i(A\in\Omega_i\;:\;A(x) \in E )\leq C\cdot Vol(E).
	\]
	Hence
	our RIFS is non-singular. The logarithmic condition \eqref{eq:CramerAss} holds for this RIFS for all $|s|$ sufficiently small by analogous reasoning to that given in Example $1$. Therefore Theorem \ref{Main theorem}, Corollary \ref{Bernoulli corollary}, and Corollary \ref{measure corollary} can be applied to this random model.
\end{example}

We conclude by remarking that to show that a RIFS is non-singular it is sufficient to show that the pushforward measure $(P_x)_{*}\eta_i$ is absolutely
continuous for all $x\in \cup_{\omega}\Pi_{\omega}(\A^{\N})$ and $i\in \A$, and that the Radon-Nikodym derivative can
be bounded above by some constant independent of $x$ and $i$. This is the technique we have used in
Example $2$.

\section{Final discussion}
\label{discussion section}
\begin{remark}
	In a random recursive model one usually expects the threshold quantities to be defined in terms of
	the arithmetic average of random variables, as opposed to the geometric average that is the expected
	behaviour in $1$-variable models. Here, this means that one na\"ively suspects the Lyapunov exponent
	to be
	\[
	\lambda'(\eta_i) = \log\int_{\Omega_i}|\Det(A)|d\eta_i(A)
	\]
	instead of 
	\[
	\lambda'(\eta_i) = \int_{\Omega_i}\log|\Det(A)|d\eta_i(A).
	\]
	While we cannot exclude the possibilities that our work could be improved, the near optimal usage of
	large deviations in our work suggests that the second Lyapunov exponent is the correct one to use.
	This is unexpected and could be explained by us requiring level specific information on worst cases, 
	as opposed to ``eventually averaging'' of behaviour of the descendants of each node.
\end{remark}

\section*{Acknowledgements}
ST was funded by the Austrian Science Fund (FWF): M-2813.

Part of this research was conducted while the authors met at the Number Theory and Dynamics
conference at the CMS, Cambridge University, in March 2019. We wish to thank the organisers
and university for the pleasant research environment.

\end{document}